\documentclass{amsart} 

\usepackage[english]{babel}
\usepackage{eurosym}
\usepackage{amsmath}
\usepackage{amsfonts}
\usepackage{amsthm}
\usepackage{amssymb}
\usepackage{mathrsfs}
\usepackage{tensor} 
\usepackage{stmaryrd} 
\usepackage{commath} 
\usepackage{calc}
\usepackage{wrapfig}
\usepackage{amscd}

\usepackage{tikz}

\newtheorem{num}{Num}[subsection]
\newtheorem{prop}[num]{Proposition}
\newtheorem{lem}[num]{Lemma}
\newtheorem{cor}[num]{Corollary}
\newtheorem{thm}[num]{Theorem}

\theoremstyle{definition}
\newtheorem{defi}[num]{Definition}

\newtheorem{rmk}[num]{Remark}
\newtheorem{eg}[num]{Example}

\newcommand{\mcl}[1]{\ensuremath{\mathcal{#1}}}
\newcommand{\mbb}[1]{\ensuremath{\mathbb{#1}}}

\newcommand{\wtl}[1]{\ensuremath{\widetilde{#1}}}

\newcommand{\prive}{\ensuremath{\backslash}}

\newcommand{\llb}{\ensuremath{\llbracket}}
\newcommand{\rrb}{\ensuremath{\rrbracket}}

\newcommand{\llpar}{(\negthinspace(}
\newcommand{\rrpar}{)\negthinspace)}

\DeclareMathOperator{\Spec}{Spec}

\DeclareMathOperator{\ord}{ord}

\newcommand{\red}{\ensuremath{\mathrm{red}}}

\newcommand{\gp}{\ensuremath{\mathrm{gp}}}

\newcommand{\sat}{\ensuremath{\mathrm{sat}}}
\newcommand{\intm}{\ensuremath{\mathrm{int}}}
\newcommand{\fs}{\ensuremath{\mathrm{fs}}}
\newcommand{\loc}{\ensuremath{\mathrm{loc}}}
\newcommand{\Hom}{\ensuremath{\mathrm{Hom}}}

\newcommand{\Gal}{\ensuremath{\mathrm{Gal}}}
\newcommand{\cX}{\mcl{X}}

\newcommand{\LL}{\mbb{L}}
\newcommand{\N}{\mbb{N}}
\newcommand{\Z}{\mbb{Z}}
\newcommand{\Q}{\mbb{Q}}
\newcommand{\R}{\mbb{R}}
\newcommand{\C}{\mbb{C}}
\newcommand{\A}{\mbb{A}}

\newcommand{\Sm}{\mathrm{Sm}}

\newcommand{\Var}{\ensuremath\mathrm{Var}}

\mathchardef\mhyp="2D 

\newcommand{\loga}[1]{\ensuremath{\mcl{#1}^{\dagger}}}

\title{Computing motivic zeta functions on log smooth models}
\author{Emmanuel Bultot}
\address{KU Leuven, Department of Mathematics, Celestijnenlaan 200B, 3001 Heverlee, Belgium}
\email{emmanuel@bultot.im }

\author{Johannes Nicaise}
\address{Imperial College,
Department of Mathematics, South Kensington Campus,
London SW7 2AZ, UK, and KU Leuven, Department of Mathematics, Celestijnenlaan 200B, 3001 Heverlee, Belgium} \email{j.nicaise@imperial.ac.uk}

\begin{document}
\begin{abstract}
We give an explicit formula for the motivic zeta function in terms of a log smooth model. It generalizes the classical formulas for snc-models, but it gives rise
to much fewer candidate poles, in general. This formula plays an essential role in recent work on motivic zeta functions of degenerating Calabi-Yau varieties
by the second-named author and his collaborators.
 As a further illustration, we explain how the formula for Newton non-degenerate polynomials can be viewed as a special case of our results.
\end{abstract}


\maketitle

\section{Introduction}
 Denef and Loeser's motivic zeta function is a subtle invariant associated with hypersurface singularities over a field $k$ of characteristic zero.
  It can be viewed as a motivic upgrade of Igusa's local zeta function for polynomials over $p$-adic fields.
 The motivic zeta function is a power series
 over a suitable Grothendieck ring of varieties, and it can be specialized to more classical invariants of the singularity, such as the Hodge spectrum. The main open problem
  in this context is the so-called {\em monodromy conjecture}, which predicts that each pole of the motivic zeta function is a root of the Bernstein polynomial of the hypersurface.

 One of the principal tools in the study of the motivic zeta function is its explicit computation on a log resolution \cite[3.3.1]{DL-barc}.
 While this formula gives a complete list of candidate poles of the zeta function, in practice most of these candidates tend to cancel out for reasons that are not well understood. Understanding this cancellation phenomenon is the key to the monodromy conjecture. The aim of this paper is to establish a formula for the motivic zeta function in terms of {\em log smooth} models instead of log resolutions (Theorem \ref{thm:main}). These log smooth models can be viewed as partial resolutions with toroidal singularities. Our formula generalizes the computation on log resolutions, but typically gives substantially fewer candidate poles (Proposition \ref{prop:poles}). A nice bonus is that, even for log resolutions, the language of log geometry allows for a cleaner and more conceptual proof of the formula for the motivic zeta function in \cite{NiSe}, and to extend the results to arbitrary characteristic (Corollary \ref{cor:snc}). We will also indicate in Section \ref{sec:curves} how our formula gives a conceptual explanation for the determination of the set of poles of the motivic zeta function of a curve singularity; this is the only dimension in which the monodromy conjecture has been proven completely.
    A special case of our formula has appeared in the literature under a different guise, namely, the calculation of motivic zeta functions of hypersurfaces that are non-degenerate with respect to their Newton polyhedron \cite{guibert}. We will explain the precise connection (and correct some small errors in \cite{guibert}) in Section \ref{sec:nondeg}.

    Our results apply not only to Denef and Loeser's motivic zeta function, but also to the motivic zeta functions of degenerations of Calabi-Yau varieties that were introduced in \cite{HaNi}. Here the formula in terms of log smooth models is particularly relevant in the context of the Gross-Siebert program on toric degenerations and Mirror Symmetry, where log smooth models appear naturally in the constructions. We have already applied our formula to compute the motivic zeta function of the degeneration of the quartic surface in \cite{NOR}. Our formula is also used in an essential way in \cite{HaNi-CY} to prove an analog of the monodromy conjecture for a large and interesting class of degenerations of Calabi-Yau varieties (namely, the degenerations with monodromy-equivariant Kulikov models).

 The main results in this paper form a part of the first author's PhD thesis \cite{PhD}. They were announced in \cite{cras}.


\subsection*{Acknowledgements} We are grateful to Wim Veys for his suggestion to interpret the results in Section \ref{sec:curves} in the context of logarithmic geometry.
  The first author was supported by a PhD grant from the Fund of Scientific Research -- Flanders (FWO).
   We would also like to thank the referee for their valuable and thoughtful comments.
  The second author was supported by the ERC Starting Grant MOTZETA (project 306610) of the European Research Council, and by long term structural funding (Methusalem
 grant) of the Flemish Government.

\subsection*{Notations and conventions}
   The general theory of logarithmic schemes that we will use is explained in \cite{kato-log,kato,GaRa}. For the reader's convenience, we have included a user-friendly introduction to regular log schemes and their fans in Section \ref{sec:logsch}. These results are not new, but they are scattered in the literature. This is not meant to be a self-contained introduction to logarithmic geometry: we assume that the reader is familiar with the basic definitions of the theory, as explained for instance in Sections 1--4 of \cite{kato-log}. In the remainder of the paper, we will frequently refer back to this section for auxiliary results on logarithmic geometry.
   
      Unless explicitly stated otherwise, all logarithmic structures in this paper are defined with respect to the Zariski topology, like in \cite{kato},
  and all the log schemes are Noetherian and fs (fine and saturated).
  This means that they satisfy condition (S) in \cite[1.5]{kato}, and are, in addition, quasi-compact.
   We will discuss a generalisation of our results to the Nisnevich topology in Section \ref{sec:etale}.

 Log schemes will be denoted by symbols of the form $\mcl{X}^\dagger$, and the underlying scheme will be denoted by $\mcl{X}$.
 We write $M_{\mcl{X}^\dagger}$ for the sheaf of monoids on  $\mcl{X}^\dagger$.
 We will follow the convention in \cite{GaRa} and speak of regular log schemes and smooth morphisms between log schemes instead of
 log regular log schemes and log smooth morphisms. When we refer to geometric properties of the underlying schemes instead, this will always be clearly indicated.

\section{Monoids}
\subsection{Generalities}
 For general background on the algebraic theory of monoids, we refer to \cite[\S4]{GaRa}.
All monoids are assumed to be commutative, and we will most often use the additive notation $(M,+)$, with neutral element $0$.
 The embedding of the category of abelian groups into the category of monoids has a left adjoint $(\cdot)^{\gp}$, which is called the groupification functor. For every monoid $M$,
 we have a canonical morphism of monoids $M\to M^{\gp}$ by adjunction. The monoid $M$ is called {\em integral} if this morphism is injective. This is equivalent to saying that the addition on $M$ satisfies the cancellation property, that is, $x+z=y+z$ implies $x=y$ for all $x$, $y$ and $z$ in $M$. An integral monoid $M$ is called {\em saturated} if an element $m$ in $M^{\gp}$ belongs to $M$ if and only if there exists an integer $d>0$ such that $dm$ belongs to $M$.  We say that a monoid $M$ is {\em fine} if it is finitely generated and integral.

 The embedding of the category of integral monoids into the category of monoids has a left adjoint, which is denoted by $(\cdot)^{\intm}$. For every monoid $M$, we have a canonical morphism $M\to M^{\intm}$ by adjunction, and this morphism is surjective. Since every group is integral, the morphism $M\to M^{\gp}$ factors through $M^{\intm}$. The induced morphism $M^{\intm}\to M^{\gp}$ is injective, so that we can identify $M^{\intm}$ with the image of the groupification morphism $M\to M^{\gp}$. Likewise, the embedding of the category of saturated monoids into the category of monoids has a left adjoint $(\cdot)^{\sat}$, and we can identify $M^{\sat}$ with the submonoid of $M^{\gp}$ consisting of the elements $m$ such that $dm$ belongs to $M^{\intm}$ for some integer $d>0$. A monoid $M$ is integral, resp.~saturated, if and only if the natural morphisms $M\to M^{\intm}$, resp.~$M\to M^{\sat}$, are isomorphisms.

  For every monoid $M$, we denote by $M^\times$ the submonoid of invertible elements of $M$. The monoid $M$ is called {\em sharp} when $M^{\times}=\{0\}$.
  Note that sharp monoids are, in particular, torsion free, because all torsion elements are invertible. We denote by $M^\sharp$ the sharp monoid  $M/M^\times$, called the {\em sharpification} of $M$. We set $M^+ = M\prive M^\times$,  the unique maximal ideal of $M$.
  For every monoid $M$, we denote by $M^\vee$ its dual monoid: $M^{\vee}=\Hom(M,\mbb{N})$. We will also consider the submonoid $M^{\vee,\loc}$ of $M^{\vee}$ consisting of local homomorphisms $M\to \mbb{N}$, that is, morphisms $\varphi:M\to \N$ such that $\varphi(m)\neq 0$ for every $m\in M^+$.

\begin{rmk}\label{rema:cones}
 When working with monoids, it is useful to keep in mind the following more concrete description: if $M$ is a fine, saturated and torsion free monoid, then we can identify $M$ with the monoid of integral points of the convex rational polyhedral cone $\sigma$ generated by $M$ in the vector space $M^{\gp}\otimes_{\Z}\R$.
  Conversely, for every convex rational polyhedral cone $\sigma$ in $\R^n$, the intersection $M=\sigma\cap \Z^n$ is a fine, saturated and torsion free monoid, by Gordan's Lemma (see \cite[4.3.22]{GaRa}).
 The monoid $M$ is sharp if and only if $\sigma$ is strictly convex. This correspondence between fine, saturated and torsion free monoids and convex rational polyhedral cones
 preserves the faces, by \cite[4.4.7]{GaRa}: the faces of $M$ are precisely the intersections of the faces of $\sigma$ with the lattice $M^{\gp}$.
\end{rmk}

A \emph{monoidal space} is a topological space $T$ endowed with a sheaf $M$ of monoids.
 If $(T,M)$ is a monoidal space, we will denote by $T^\sharp$ the monoidal space obtained by equipping  $T$ with the sheaf $M^\sharp = M/M^\times$, where
\[M^\times(U)= \bigl(M(U)\bigr)^\times\] for every open subspace $U$ in $T$. This construction applies, in particular, to
 the monoidal space $(\mcl{X},M_{\loga{X}})$ associated with a log scheme $\loga{X}$; the result will be denoted by $(\loga{X})^{\sharp}$.


\subsection{The root index}

\begin{defi}
 Let $M$ be a fine and saturated monoid endowed with a morphism of monoids $\varphi:\mbb{N}\to M$. The {\em root index} of $\varphi$ is defined to be $0$ if $\varphi(1)$ is invertible. Otherwise,
 it is the largest positive integer $\rho$ such that the residue class of $\varphi(1)$ in $M^{\sharp}$ is divisible by $\rho$.
\end{defi}

Note that such a largest $\rho$ exists because $M^{\sharp}$ is a submonoid of the free abelian group of finite rank $(M^{\sharp})^{\gp}$.
 The importance of the root index lies in the following properties.
 \begin{prop}\label{prop:root}
 Let $M$ be a fine and saturated monoid, and let $\varphi:\mbb{N}\to M$ be a morphism of monoids. Denote by $\rho$ the root index of $\varphi$.
  For every positive integer $d$, we consider the monoid
  $$M(d)=\left(M\oplus_{\mbb{N}}\frac{1}{d}\mbb{N} \right).$$
 \begin{enumerate}
 \item \label{it:rootfine} The monoid $M(d)$ is fine for every $d>0$.

 \item \label{it:rootdiv} If $d$ divides $\rho$, then
 $M^{\sharp}\to (M(d)^{\sat})^{\sharp}$ is an isomorphism,
 and the morphism
 $$\varphi_d:\frac{1}{d}\mbb{N}\to M(d)^{\sat}$$
 has root index $\rho/d$.

 \item \label{it:rootsharp} If $d$ is prime to $\rho$, then the morphisms
$$M^{\times}\to M(d)^{\times},\qquad M(d)^{\times}\to (M(d)^{\sat})^{\times}$$
 are isomorphisms. In particular, if $M$ is sharp, then so are $M(d)$ and $M(d)^{\sat}$.
  \end{enumerate}
 \end{prop}
 \begin{proof}
 \eqref{it:rootfine} It is obvious that $M(d)$ is finitely generated. It is also integral because we can apply the criteria of \cite[4.1]{kato-log} to the morphism $\N\to (1/d)\N$.

 \eqref{it:rootdiv}
   Assume that $d$ divides $\rho$. We may suppose that $M$ is sharp, because the morphism
   $$(M(d)^{\sat})^{\sharp}\to (M^{\sharp}(d)^{\sat})^{\sharp}$$
   is an isomorphism by the universal properties of sharpification, coproduct and saturation.
       Let $m$ be an element of $M$ such that $\rho m=\varphi(1)$.
  Then $$d((\rho/d)m-\varphi_d(1/d))=0,$$ so that $(\rho/d)m-\varphi_d(1/d)$ is a unit in $M(d)$ and
  $(\rho/d)m=\varphi_d(1/d)$ in $(M(d)^{\sat})^{\sharp}$. Using the universal properties of the coproduct, saturation and sharpification, together with the fact that $M$ is sharp and saturated, we obtain a morphism of monoids $(M(d)^{\sat})^{\sharp}\to M$ that sends the residue class of $\varphi_d(1/d)$ to $(\rho/d)m$ and that is inverse to   $M\to (M(d)^{\sat})^{\sharp}$. It follows at once that
 $$\varphi_d:\frac{1}{d}\mbb{N}\to M(d)^{\sat}$$
 has root index $\rho/d$.

\eqref{it:rootsharp} Assume that $d$ is prime to $\rho$.
 It suffices to prove that the composed morphism $M^{\times}\to (M(d)^{\sat})^{\times}$ is an isomorphism, because
 $M(d)^{\times}\to (M(d)^{\sat})^{\times}$ is injective since $M(d)$ is integral.
 Let $x$ be an invertible element in $M(d)^{\sat}$. We must show that $x$ lies in $M$.
  Since $M\to M(d)^{\sat}$ is exact by \cite[4.4.42(vi)]{GaRa}, it is enough to prove that $x\in M^{\gp}$.
 We can write $x$ as $(m,i/d)$ with $m\in M^{\gp}$ and $i\in \{0,\ldots,d-1\}$.
 Since $M$ is saturated, the element $dx$ lies in $M$, and hence in $M^{\times}$ because we can apply the same argument to the inverse of $x$. This means that $dm=-i\varphi(1)$ in $M^{\sharp}$.
 But $d$ is prime to $\rho$, so that $d$ divides $i$.
 Hence, $i=0$ and $x\in M^{\gp}$.
 \end{proof}

\section{Regular log schemes}\label{sec:logsch}
\subsection{Kato's definition of regularity}
The notion of regularity for logarithmic schemes was introduced by K.~Kato in \cite{kato}. It can be viewed as a generalization of the theory of toroidal embeddings
 in \cite{KKMS}. An important advantage of the logarithmic approach is that it works equally well in mixed characteristic; moreover, the monoidal structure on logarithmic schemes
 keeps track in an efficient way of the cones that describe the local toroidal structure.

 Let $\mcl{X}^{\dagger}$ be a log scheme. Thus $\loga{X}$ consists of a scheme $\mcl{X}$, equipped with a sheaf of monoids
 $M_{\loga{X}}$ and a morphism of sheaves of monoids $$\alpha\colon (M_{\loga{X}},+)\to (\mathcal{O}_{\mcl{X}},\times)$$ such that the induced morphism $\alpha^{-1}(\mathcal{O}^{\times}_{\mcl{X}})\to \mathcal{O}^{\times}_{\mcl{X}}$ is an isomorphism. In particular, $\alpha$ identifies the invertible elements in $M_{\loga{X}}$ and $\mathcal{O}_{\mcl{X}}$. For every point $x$ of $\mcl{X}$, we call the monoid $$M^{\sharp}_{\loga{X},x}=M_{\loga{X},x}/M^{\times}_{\loga{X},x}\cong M_{\loga{X},x}/\mathcal{O}^{\times}_{\mcl{X},x}$$ the {\em characteristic monoid} of $\loga{X}$ at $x$.
  We say that the log structure is {\em trivial} at $x$ if $M^{\sharp}_{\loga{X},x}=\{0\}$, that is, if every element in the stalk $M_{\loga{X},x}$ is invertible.

 Assume that $\loga{X}$ is Noetherian and fs (fine and saturated). Thus it satisfies condition (S) in \cite[1.5]{kato}. Then the characteristic monoids of $\loga{X}$ are sharp, fine and saturated.
 For every point $x$ of the underlying scheme $\mcl{X}$, we denote by $I_{\mcl{X}^{\dagger},x}$ the ideal in $\mathcal{O}_{\mcl{X},x}$ generated by the image of
 $M^+_{\loga{X},x}=M_{\loga{X},x}\setminus M^{\times}_{\loga{X},x}$ in $\mathcal{O}_{\mcl{X},x}$.
 The log scheme $\loga{X}$ is called {\em regular} at $x$ if $\mathcal{O}_{\mcl{X},x}/I_{\loga{X},x}$ is regular and
 $$\dim(\mathcal{O}_{\mcl{X},x})=\dim(\mathcal{O}_{\mcl{X},x}/I_{\loga{X},x})+\dim(M^{\sharp}_{\loga{X},x}).$$
  Here $\dim(M^{\sharp}_{\loga{X},x})$ denotes the Krull dimension of the monoid $M^{\sharp}_{\loga{X},x}$, which can also be expressed as the rank of the free abelian group $(M^{\sharp}_{\loga{X},x})^{\gp}$ by \cite[4.4.10]{GaRa}.
  An important consequence of the definition is that every regular log scheme is normal, by \cite[4.1]{kato}. Every quasi-compact fine and saturated log scheme that is smooth over a regular log scheme is itself regular, by \cite[8.2]{kato}.

 The idea behind the definition of regularity for log schemes is that, when $\mcl{X}^{\dagger}$ is regular at $x$, the lack of regularity of the scheme $\mcl{X}$ at $x$ is encoded in the characteristic monoid $M^{\sharp}_{\mcl{X}^{\dagger},x}$.
 By Remark, \ref{rema:cones}, we can think of $M^{\sharp}_{\loga{X},x}$ as the monoid of integral points in a strictly convex rational polyhedral cone, and this cone describes the toroidal structure of $\mcl{X}$ at $x$. An explicit description of the completed local ring of $\mcl{X}$ at $x$ in terms of the characteristic monoid can be found in \cite[3.2]{kato}.

\begin{eg}\label{ex:snc}
An important class of logarithmic structures are the {\em divisorial} log structures. Let $\cX$ be a Noetherian integral scheme, and let $D$ be a Weil divisor on $\cX$. Then $D$ gives rise to a log structure $M$ on $\cX$, called the divisorial log structure induced by $D$. The sheaf of monoids $M$ is defined by
$$M=\mathcal{O}_{\mcl{X}}\cap i_* \mathcal{O}^{\times}_{\mcl{X}\setminus D} $$
where $i\colon \mcl{X}\setminus D\to \mcl{X}$ is the open embedding of the complement of $D$ in $\mcl{X}$. Thus the sections of $M$ are the regular functions on $\mcl{X}$ that are invertible outside $D$.

The log scheme $\mcl{X}^{\dagger}=(\mcl{X},M)$ is not fine and saturated, in general. However, if $\mcl{X}$ is regular and the reduced divisor $D_{\red}$ has strict normal crossings, then
$\mcl{X}^{\dagger}$ is regular. If $x$ is a point of $\mcl{X}$ and $(z_1,\ldots,z_n)$ is a regular system of local parameters in $\mathcal{O}_{\mcl{X},x}$ such that $D_{\red}$ is defined by $z_1\cdot \ldots \cdot  z_r=0$ locally at $x$, for some $0\leq r\leq n$, then $I_{\mcl{X}^\dagger,x}$ is the ideal generated by $(z_1,\ldots,z_r)$, and
$M_x/M^{\times}_x$ is a free monoid of rank $r$. A basis for this monoid is given by the residue classes of $z_1,\ldots,z_r$.

Conversely, let $\mcl{X}^{\dagger}$ be a log regular scheme. Let $D$ be the set of the points $x$ of $\mcl{X}$ where the log structure is nontrivial. It follows from \cite[11.6]{kato} that $D$ is a closed subset of $\mcl{X}$ of pure codimension one; thus we can view it
as a reduced Weil divisor on $\mcl{X}$. Then by \cite[11.6]{kato}, the log structure on $\mcl{X}^{\dagger}$ is precisely the divisorial log structure induced by $D$. For every point $x$ of $\loga{X}$, we can interpret the characteristic monoid $$M^{\sharp}_{\loga{X},x}=M_{\loga{X},x}/\mathcal{O}^{\times}_{\mcl{X},x}$$
 as the monoid of effective Cartier divisors on $\Spec \mathcal{O}_{\mcl{X},x}$ supported on the inverse image of $D$.

 Note, however, that $\mcl{X}$ is not necessarily regular. For instance, if $\mcl{X}$ is a toric variety over a field, and $D$ is its toric boundary, then the divisorial log structure induced by $D$ makes $\mcl{X}$ into a regular log scheme (see Example \ref{exam:toric}).
\end{eg}

\subsection{Fans and log stratifications}\label{sec:fans}
Let $\mcl{X}^\dagger$ be a regular log scheme, and consider its associated fan $F(\mcl{X}^\dagger)$ in the sense of \cite[\S10]{kato}. This is a sharp monoidal space whose underlying topological space is the subspace of $\mcl{X}$ consisting of the points $x$ such that $M_{\cX^\dagger,x}^+$ generates the maximal ideal of $\mathcal{O}_{\cX,x}$. The sheaf of monoids on $F(\mcl{X}^\dagger)$ is the pullback of the sheaf $M^{\sharp}_{\loga{X}}=M_{\mcl{X}^\dagger}/\mathcal{O}_{\mcl{X}}^{\times}$ on $\mcl{X}$.
 A dictionary between this notion of fan and the usual notion in toric geometry is provided in Example \ref{exam:toric} below.

  By \cite[10.6.9(ii)]{GaRa}, the natural morphism of monoidal spaces
  $$F(\mcl{X}^\dagger)\to (\loga{X})^{\sharp}=(\mcl{X},M^{\sharp}_{\loga{X}})$$
has  a canonical retraction $\pi\colon  (\loga{X})^{\sharp}\to F(\mcl{X}^\dagger)$. It sends a point $x$ of $\mcl{X}$ to the point of $F(\mcl{X}^\dagger)$ that corresponds to the prime ideal
$I_{\loga{X},x}$ of $\mathcal{O}_{\mcl{X},x}$. The map $\pi$ is open, by \cite[10.6.9(iii)]{GaRa}.
 With the help of this retraction, we can enhance the construction of the Kato fan to a functor from the category of regular log schemes to the category of fans, by sending
 a morphism of regular log schemes $h:\loga{Y}\to \loga{X}$ to the morphism of fans obtained as the composition
 $$\begin{CD}F(\loga{Y})@>>> (\loga{Y})^{\sharp}@>h^{\sharp}>> (\loga{X})^{\sharp}@>\pi>> F(\loga{X}).  \end{CD}$$
See \cite[\S10.6]{GaRa} for additional background.

 For every point $\tau$ of $F(\loga{X})$, we denote by $r(\tau)$ the dimension of the sharp, fine and saturated monoid $M_{F(\loga{X}),\tau}=M^{\sharp}_{\loga{X},\tau}$. We call this number the {\em rank} of $\tau$.
   The fiber $\pi^{-1}(\tau)$ is an irreducible locally closed subset of $\mcl{X}$ of pure codimension $r(\tau)$,
  and it is a regular subscheme of $\mcl{X}$ if we endow it with its reduced induced structure \cite[10.6.9(iii)]{GaRa}.
      We denote this subscheme by $E(\tau)^o$.  Locally around each point $x$ of $E(\tau)^o$, the scheme $E(\tau)^o$ is the zero locus of the prime ideal $I_{\loga{X},x}$ of $\mathcal{O}_{\mcl{X},x}$.
      We write $E(\tau)$ for the schematic closure of $E(\tau)^o$ in $\mcl{X}$.
    Then $E(\tau)$ is the disjoint union of the sets $E(\sigma)^o$ where $\sigma$ runs through the closure of $\{\tau\}$ in $F(\loga{X})$, because the retraction $\pi$ is open.
     Thus the collection of subschemes $$\{E(\tau)^o\,|\,\tau\in F(\loga{X})\}$$ is a stratification of $\mcl{X}$, which is called the log stratification of $\loga{X}$.
  It follows immediately from the definitions that $\tau$ is the generic point of $E(\tau)^o$ and $E(\tau)$.

  \begin{prop}\label{prop:cosp}
  Let $\loga{X}$ be a regular log scheme.
  Then the morphism
  $$\pi^{-1}M_{F(\mcl{X}^\dagger)}\to M^{\sharp}_{\loga{X}}$$ is an isomorphism of sheaves of monoids. In particular, the sheaf
  $M^{\sharp}_{\loga{X}}$ is constant along every stratum $E(\tau)^o$ of the logarithmic stratification.
 \end{prop}
\begin{proof}
This is stated without proof in \cite[10.2]{kato}; as observed in \cite[10.6.9(i)]{GaRa}, it follows from the construction of the retraction $\pi$ in \cite[10.6.9(ii)]{GaRa}.
\end{proof}
It follows that, for every point $x$ of $E(\tau)^o$, the monoid $M_{\loga{X},x}^{\sharp}$ has dimension $r(\tau)$; we say that $E(\tau)^o$ is a stratum of rank $r(\tau)$.

\begin{eg}\label{exam:standard}
Let $R$ be a discrete valuation ring.
 We write $S=\Spec R$ and denote by $s$ and $\eta$ the closed and generic point of $S$, respectively.
 We denote by $S^{\dagger}$ the log scheme obtained by endowing $S$ with the divisorial log structure induced by the closed point $s$; this is called the {\em standard} log structure on $S$. Then $M_{S^{\dagger}}(S)=R\setminus \{0\}$ and $M_{S^{\dagger}}(\eta)=K^{\times}$.
 Thus the fan $F(S^{\dagger})$ consists of the underlying topological space $|S|=\{s,\eta\}$ of $S$, endowed with the sheaf of monoids
 $M_{F(S^{\dagger})}$ determined by $$M_{F(S^{\dagger})}(S)=(R\setminus \{0\})/R^{\times}=\N, \quad M_{F(S^{\dagger})}(\eta)=\{0\}.$$
\end{eg}

\begin{eg}\label{exam:toric}
Let $k$ be a field, and let $Y=Y(\Sigma)$ be a toric variety over $k$, associated with a fan $\Sigma$ in $\R^n$.
 We endow $Y$ with the divisorial log structure induced by the toric boundary divisor $D$; the resulting log scheme will be denoted by $Y^{\dagger}$.

 Let $y$ be a point of $Y$, and let $\sigma$ be the unique cone in $\Sigma$ such that $y$ is contained in the torus orbit $O(\sigma)$.
 We will describe the characteristic monoid $M^{\sharp}_{\loga{Y},y}$ at $y$.  By definition, the monoid $M_{\loga{Y},y}$ consists of the functions in $\mathcal{O}_{Y,y}$ that
are invertible outside $D$.
 By \cite[3.3]{fulton}, we can write such a function in a unique way as $u\chi^m$, where $u$ is a unit in $\mathcal{O}_{Y,y}$ and $\chi^m$ is the character associated with an element $m$ of $\sigma^{\vee}\cap \Z^n$. Thus $M^{\sharp}_{\loga{Y},y}=(\sigma^{\vee}\cap \Z^n)^{\sharp}$. Locally around $y$, the zero locus of the elements in $M^+_{\loga{Y},y}$ is the torus orbit $O(\sigma)$. Since $O(\sigma)$ is regular of dimension $$\dim(Y)-\dim(\sigma)= \dim(Y)-\dim(M^{\sharp}_{\loga{Y},y}),$$
 the log scheme $\loga{Y}$ is regular at $y$.

 Our description of the characteristic monoids of $\loga{Y}$ also implies that the logarithmic strata of $\loga{Y}$ are precisely the torus orbits of $Y$.
  Thus the points of the fan $F(\loga{Y})$ are in bijective correspondence with the cones in the fan $\Sigma$.
 If $y$ and $y'$ are points of $F(\loga{Y})$ and we denote by $\sigma$ and $\sigma'$ the corresponding cones in $\Sigma$, then $y'$ lies in the closure of $\{y\}$ if and only if
 $\sigma$ is a face of $\sigma'$. The cospecialization map
 $$M_{F(\loga{Y}),y'}\to  M_{F(\loga{Y}),y}$$ is precisely the inclusion
 $$((\sigma')^{\vee}\cap \Z^n)^\sharp\to (\sigma^{\vee}\cap \Z^n)^\sharp$$ induced by the face map $\sigma\to \sigma'$.
 Thus the fan $F(\loga{Y})$ encodes the monoids of integral points in the cones of $\Sigma$ (or rather, the sharpified dual cones) and how they are glued together in the fan $\Sigma$; see also \cite[9.5]{kato}.
\end{eg}

\subsection{Boundary divisor and divisorial valuations}\label{sec:val}
 Let $\loga{X}$ be a regular log scheme. The {\em boundary} of $\loga{X}$ is the locus of points where the log structure is non-trivial; this is a closed subset of $\mcl{X}$ of pure codimension one, and it is equal to the union of the strata
 $E(\tau)^o$ such that $M_{F(\loga{X}),\tau}\neq 0$.
  We denote by $D$ the reduced Weil divisor supported on the boundary of $\loga{X}$.
 Then the log structure on $\loga{X}$ coincides with the divisorial log structure induced by $D$ (see Example \ref{ex:snc}).
  If we denote by $E_i,\,i\in I$ the prime components of $D$, then
 a point $\tau$ of $\mcl{X}$ belongs to $F(\loga{X})$ if and only if it is a generic point of $\cap_{j\in J}E_j$ for some subset $J$ of $I$, by \cite[2.2.3]{BM} (note that, when $J$ is empty, this intersection equals $\mcl{X}$). In that case, $E(\tau)$ is
 the connected component of $\cap_{j\in J}E_j$ that contains $\tau$, and
   $$E(\tau)^o=E(\tau)\setminus \bigcup_{i\notin J}E_i.$$

   \begin{eg}
Let $\cX$ be a quasi-compact regular scheme and let $D$ be a strict normal crossings divisor on $\cX$.
 Let $\loga{X}$ be the log scheme that we obtain by endowing $\cX$ with the divisorial log structure induced by $D$ (see Example \ref{ex:snc}). Then $\loga{X}$ is log regular and its boundary
 coincides with $D$. Conversely, if $\loga{X}$ is a regular log scheme, then the underlying scheme $\mcl{X}$ is regular if and only if $M_{F(\loga{X}),\tau}$ is isomorphic to $\N^{r(\tau)}$  for every $\tau$ in $F(\loga{X})$ \cite[10.5.35]{GaRa}. In that case, the boundary divisor $D$ of $\loga{X}$ has strict normal crossings \cite[2.7]{saito}.
\end{eg}

 Let $\loga{X}$ be any regular log scheme, and fix a point $\tau$ of the fan $F(\loga{X})$. Let $F(\tau)$ be the subspace of $F(\loga{X})$ consisting of the points $\sigma$ such that $\tau$ is contained in the closure of
  $\{\sigma\}$. We denote by $M_{F(\tau)}$ the restriction of the sheaf of monoids $M_{F(\loga{X})}$ to $F(\tau)$. Then the monoidal space $(F(\tau),M_{F(\tau)})$ is canonically isomorphic to the spectrum of the characteristic monoid $M_{F(\loga{X}),\tau}$, by \cite[10.1]{kato} and its proof.
    This implies, in particular, that the prime ideals of height one in $M_{F(\loga{X}),\tau}$ are in bijective correspondence with the
    strata $E(\sigma)$ such that $\sigma$ is a codimension one point in $F(\loga{X})$ whose closure contains $\tau$;
  these are precisely the irreducible components of $D$ that pass through $\tau$.

  If $\mathfrak{p}$ is a prime ideal of height one in $M_{F(\loga{X}),\tau}$ and $\sigma$ is the corresponding point of $F(\tau)$, then
  $$M_{F(\loga{X}),\sigma}\cong (M_{F(\loga{X}),\tau})_{\mathfrak{p}}^{\sharp}= \N.$$
   This monoid is generated by the residue class of a local equation for $E(\sigma)$ at its generic point $\sigma$.
  More generally, for every element $f$ in $M_{\loga{X},\sigma}$, the residue class of $f$ in $M_{F(\loga{X}),\sigma}= \N$
  is the order of $f$ along the component $E(\sigma)$.

  We can also interpret this from the perspective of the dual monoid $M^{\vee}_{F(\loga{X}),\tau}$.
   The complement of a height one prime ideal $\mathfrak{p}$ in $M_{F(\loga{X}),\tau}$
   is a codimension one face of $M_{F(\loga{X}),\tau}$. The dual face
   $$(M_{F(\loga{X}),\tau}\setminus \mathfrak{p})^{\bot}= \{\varphi\in M^{\vee}_{F(\loga{X}),\tau}\,|\,\varphi(m)=0 \mbox{ for all }m\notin \mathfrak{p}\}$$
   of $M^{\vee}_{F(\loga{X}),\tau}$ is generated by
   the localization morphism $$M_{F(\loga{X}),\tau}\to (M_{F(\loga{X}),\tau})_{\mathfrak{p}}^{\sharp}= \N.$$
      The map
 $$\mathfrak{p}\mapsto (M_{F(\loga{X}),\tau}\setminus \mathfrak{p})^{\bot}$$
  is a bijection between the set of height one prime ideals in $M_{F(\loga{X}),\tau}$ and the set of one-dimensional faces
  of $M^{\vee}_{F(\loga{X}),\tau}$ (in view of Remark \ref{rema:cones}, this follows from the classical duality properties of
  convex rational polyhedral cones).

\subsection{Subdivisions}\label{sec:subdiv}
 Let $F'$ be a fine and saturated proper subdivision of the fan $F(\loga{X})$ in the sense of \cite[9.7]{kato}.
 Such a subdivision gives rise to a proper \'etale morphism of log schemes $h:\loga{Y}\to \loga{X}$
 such that $F(\loga{Y})$ is isomorphic to $F'$ over $F(\loga{X})$. Moreover, $h$ is an isomorphism over
  the log trivial locus of $\loga{X}$. More precisely, the discriminant locus of $h:\mcl{Y}\to \mcl{X}$ is the union of
  strata $E(\tau)$ in $\mcl{X}$ such that the morphism of monoidal spaces $F'\to F(\loga{X})$ is not an isomorphism over any open neighbourhood of $\tau$ in $F(\loga{X})$.
 Let $\tau'$ be a point of $F(\loga{Y})$ and denote by $\tau$ its image in $F(\loga{X})$. Then the morphism $M^{\gp}_{F(\loga{X}),\tau}\to M^{\gp}_{F(\loga{Y}),\tau'}$ is surjective by the definition of a subdivision. In particular, $r(\tau)\geq r(\tau')$.
  Denote by $N$ the kernel of $M^{\gp}_{F(\loga{X}),\tau}\to M^{\gp}_{F(\loga{Y}),\tau'}$.
   It follows immediately from the construction of
   the morphism $h\colon \loga{Y}\to \loga{X}$ in the proof of \cite[9.9]{kato}
   that $E(\tau')^o$ is
 a torsor over $E(\tau)^o$ with translation group $$\Spec \Z[N]\cong \mathbb{G}^{r(\tau)-r(\tau')}_{m,\Z}.$$

\begin{eg}
For toric varieties, the subdivisions in \cite[9.6]{kato} correspond to subdivisions of the toric fan {\em via} the dictionary
 provided in Example \ref{exam:toric}. The associated morphism of log schemes is precisely the toric modification associated with
 the subdivision of the toric fan.
\end{eg}

\subsection{Charts}\label{sec:charts}
Let $\loga{X}$ be a fine and saturated log scheme. A {\em chart} for $\loga{X}$ is a strict morphism of log schemes $\loga{X}\to \Spec^\dagger \Z[N]$, where $N$ is a monoid
and we denote by $\Spec^{\dagger}\Z[N]$ the scheme $\Spec \Z[N]$ endowed with the log structure induced by $N\to \Z[N]$. Here {\em strict} means that the log structure on $\loga{X}$ is the pullback of the log structure on $\Spec^{\dagger} \Z[N]$. If $\loga{X}\to \Spec^{\dagger} \Z[N]$ is a chart for $\loga{X}$, then for every point $x$ of $\mcl{X}$, the morphism of monoids $N\to M_{\loga{X},x}$ induces a surjection $N\to M^{\sharp}_{\loga{X},x}$. Thus, up to a multiplicative factor, every element in $M^+_{\loga{X},x}$ lifts to an element of $N^+$.
 Therefore, if $\loga{X}$ is regular, then locally around $x$, the logarithmic stratum that contains $x$ is the zero locus of the image of $N^+$ in $\mathcal{O}_{\mcl{X},x}$. We will use this description in the proof of Lemma \ref{lemm:root}.

  For every fine and saturated log scheme $\loga{X}$ and every point $x$ of $\mcl{X}$, we can find locally around $x$ a chart of the form $\loga{X}\to \Spec^{\dagger} \Z[M^{\sharp}_{\loga{X},x}]$, by  the proof of \cite[10.1.36(i)]{GaRa}. Then the induced morphism of monoids $M^{\sharp}_{\loga{X},x}\to M_{\loga{X},x}$ is a section of the projection morphism
$M_{\loga{X},x}\to M^{\sharp}_{\loga{X},x}$.

 If $f\colon \loga{Y}\to \loga{X}$ is a morphism of fine and saturated log schemes, then a chart for $f$ is a commutative diagram of log schemes
 $$\begin{CD}
 \loga{Y}@>>> \Spec^{\dagger}\Z[P]
 \\ @VfVV @VVgV
 \\ \loga{X} @>>> \Spec^{\dagger}\Z[N]
 \end{CD}$$
where $g$ is induced by a morphism of monoids $N\to P$ and the horizontal maps are charts. Locally on $\loga{Y}$ and $\loga{X}$, we can
always find a chart for $f$ with $N$ and $P$ fine and saturated monoids. More precisely, if $\loga{X}\to \Spec^{\dagger}\Z[N]$ is any chart with $N$ fine and saturated,
then locally on $\loga{Y}$, we can find a morphism of fine and saturated monoids $N\to P$ and a chart for $f$ as in the diagram above. This follows from \cite[10.1.40]{GaRa}, except for the fact that we can choose $P$ to be saturated; that property follows from the proof of  \cite[10.1.37]{GaRa}.

One usually cannot choose the chart
 $\loga{Y}\to \Spec^{\dagger}\Z[P]$ independently of the morphism $f$, because there may not be a morphism of monoids $N\to P$ that makes the diagram commute.
 In practice, one can solve this problem by modifying the monoid $P$; let us discuss the only case that will be used in this paper (in the proof of Lemma \ref{lemm:root}).
  Suppose that we are given charts $\loga{X} \to \Spec^{\dagger}\Z[N]$ and $\loga{Y} \to \Spec^{\dagger}\Z[Q]$, where the monoid $N$ is isomorphic to $\N$.
  Let $y$ be a point of $\mcl{Y}$ and set $x=f(y)$. We denote by $\varphi$ the morphism of monoids $Q\to M_{\loga{Y},y}$ induced by the given chart for $\loga{Y}$.
  Denote by $h$ the image of the generator of $N$ under the composed morphism
  $$N\to M_{\loga{X},x}\to M_{\loga{Y},y}.$$ By the surjectivity of the morphism $Q\to M^{\sharp}_{\loga{Y},y}$, we can find a unit $u$
  in $\mathcal{O}_{\mcl{Y},y}$ such that $uh$ lifts to an element $q$ in $Q$. Now we set $P=Q\times \Z$ and, locally around $y$, we consider the chart for $\loga{Y}$ induced by the morphism
  of monoids $P\to M_{\loga{Y},y}$ that maps $(a,b)$ to $u^b\varphi(a)$. Then $f$ has a chart as in the diagram above, where the morphism of monoids
  $N\to P$ sends the generator of $N$ to $(q,-1)$.

  \subsection{Smoothness versus regularity}
 For the applications in Section \ref{sec:DL} we will also need the following result, which relates log regularity and log smoothness.
 For the definition of regularity for log schemes with respect to the \'etale topology, we refer to \cite[2.2]{niziol}. It is the direct analog
 of the Zariski case, replacing points by geometric points and local rings by their strict henselizations. The basic theory of logarithmic smoothness can be found in \cite[\S3.3]{kato-log}.

 \begin{prop}\label{prop:regsm}
 Let $R$ be a discrete valuation ring with residue field $k$, and set $S=\Spec(R)$. We denote by $S^{\dagger}$ the scheme $S$ equipped with its standard log structure (see Example \ref{exam:standard}).
Let $\loga{X}$ be a fine and saturated log scheme of finite type over $S^{\dagger}$ (with respect to the Zariski or \'etale topology).
\begin{enumerate}
\item \label{it:flat} If $\loga{X}$ is regular, then $\mcl{X}$ is flat over $S$.
\item \label{it:smreg} If $\loga{X}$ is smooth over $S^{\dagger}$, then it is regular.
\item \label{it:regsm} If $k$ has characteristic zero and $\loga{X}$ is regular, then $\loga{X}$ is smooth over $S^{\dagger}$ .
\end{enumerate}
 \end{prop}
 \begin{proof}
\eqref{it:flat} Assume that $\loga{X}$ is regular. Then $\loga{X}$ is normal; in particular, it does not have any embedded components. For every point $x$ of the special fiber $\mcl{X}_k$, the morphism
$\loga{X}\to S^{\dagger}$ induces a commutative diagram
$$\begin{CD}
M_{\loga{X},x}@>>>\mathcal{O}_{\loga{X},x}
\\ @AAA @AAA
\\ M_{S^{\dagger},s}=R\setminus \{0\}@>>>R
\end{CD}$$
where $s$ denotes the closed point of $S$. If $t$ is a uniformizer in $R$, then $t$ is not invertible in
$\mathcal{O}_{\loga{X},x}$, so that the log structure on $\loga{X}$ is non-trivial at every point of $\mcl{X}_k$.
 However, the definition of regularity implies that the log structure is trivial at every generic point of $\mcl{X}$; thus all the generic points of $\mcl{X}$ are contained in the generic fiber, so that $\mcl{X}$ is flat over $S$.

\eqref{it:smreg} Every smooth fs log scheme over $S^{\dagger}$ is regular by \cite[8.2]{kato} (we can reduce to the Zariski case by passing to an \'etale cover of $\mcl{X}$ where the log structure on $\loga{X}$ is Zariski in the sense of \cite[2.1.1]{niziol}).

\eqref{it:regsm} Assume that $k$ has characteristic zero, and that $\loga{X}$ is regular. Let  $\overline{x}$ be a geometric point on $\mcl{X}_k$, and set $M=M^{\sharp}_{\loga{X},\overline{x}}$.
 We choose a chart $\loga{X}\to \Spec^{\dagger}\Z[M]$ \'etale-locally around $\overline{x}$. We also choose a uniformizer $t$ in $R$; this choice determines a chart $S^{\dagger}\to \Spec^{\dagger}\Z[\N]$ such that the induced morphism $\N\to M_{S^{\dagger},s}$ maps $1$ to $t$. Then we can find an element $m$ in $M$ and a unit $u$ in $\mathcal{O}_{\mcl{X},\overline{x}}$ such that
 $t=um$ in $\mathcal{O}_{\mcl{X},\overline{x}}$.
  Since $k$ has characteristic zero, we can take arbitrary roots of invertible functions on $\mcl{X}$ locally in the \'etale topology on $\mcl{X}$. Thus there exists a morphism $\psi$ from the free abelian group $M^{\gp}$ to $\mathcal{O}^{\times}_{\mcl{X},\overline{x}}$
  that maps $m$ to $u$. Multiplying the morphism of monoids $M\to \mathcal{O}_{\mcl{X},\overline{x}}$ with the restriction of $\psi$ to $M$, we obtain a new chart $\loga{X}\to \Spec^{\dagger}\Z[M]$ \'etale-locally around $\overline{x}$ such that the pullback of $m$ is equal to $t$.

  This implies that, \'etale locally around every geometric point $\overline{x}$ on $\mcl{X}_k$, we can find a chart for $\loga{X}\to S^{\dagger}$ of the form
        $$\begin{CD}
      \loga{X}@>>> \Spec^{\dagger} \Z[M]
       \\ @VVV @VVV
     \\  S^{\dagger}@>>> \Spec^{\dagger}  \Z[\N]
       \end{CD}$$
 The morphism $\N\to M$ is injective, because $M$ is integral and  $t$ is not invertible at $\overline{x}$.
  The order of the torsion part of  $\mathrm{coker}(\N^{\gp}=\Z\to M^{\gp})$ is invertible in $k$, by our assumption that $k$ has characteristic zero.
 Moreover, the morphism of schemes
 $$\mcl{X}\to \Spec \Z[M] \times_{\Spec \Z[\N]}S$$
 is smooth over a neighbourhood of $\overline{x}$, by the local description of regular log schemes in \cite[3.2(1)]{kato}.
  Thus it follows from Kato's logarithmic  criterion for smoothness \cite[3.5]{kato-log} that $\loga{X}\to S^{\dagger}$ is smooth.
 \end{proof}

Beware that Proposition \ref{prop:regsm}\eqref{it:regsm} does not extend to the case where $k$ has characteristic $p>0$.  The problem is that we cannot take $p$-th roots of
all invertible functions locally in the \'etale topology, and that the order of the torsion part of the cokernel of the morphism $\Z\to M^{\gp}$ may not be invertible in $k$. A sufficient condition for log smoothness is given by the following statement. Let $\loga{X}$ be a regular log scheme of finite type over $S^{\dagger}$ (with respect to the \'etale topology). Suppose moreover that $k$ is perfect,
 the log structure on $\loga{X}$ is vertical (that is, it is trivial on $\mcl{X}_K$),
 the generic fiber $\mcl{X}_K$ is smooth over $K$, and the multiplicities of the components in the special fiber are prime to $p$. Then $\loga{X}$ is smooth over $S^{\dagger}$. This follows from the same argument as in the proof of Proposition \ref{prop:regsm}.

\subsection{Fine and saturated fibered products}\label{sec:fsfib}
An important role in this paper is played by fibered products in the category of fine and saturated log schemes. Let us briefly describe their structure for further reference.
 Let $\loga{X}\to \loga{Z}$ and $\loga{Y}\to \loga{Z}$ be morphisms of fine and saturated log schemes.
  The fibered product $\loga{W}$ of $\loga{X}$ and $\loga{Y}$ over $\loga{Z}$ in the category of log schemes
  can be constructed by endowing the scheme $\mcl{W}=\mcl{X}\times_{\mcl{Z}}\mcl{Y}$ with the fibered coproduct of the
  pullbacks of the log structures on $\loga{X}$ and $\loga{Y}$ over the pullback of the log structure on $\loga{Z}$ (see \cite[1.6]{kato-log}).
 Let $w$ be a point of $\mcl{W}$ that maps to $x$, $y$ and $z$ in $\mcl{X}$, $\mcl{Y}$ and $\mcl{Z}$, respectively. Then it follows directly from
 the construction that the characteristic monoid $M^{\sharp}_{\loga{W},w}$ is canonically isomorphic to
 $$\left(M^{\sharp}_{\loga{X},x}\oplus_{M^{\sharp}_{\loga{Z},z}} M^{\sharp}_{\loga{Y},y}\right)^{\sharp}.$$
  If  $M^{\sharp}_{\loga{Z},z}\to M^{\sharp}_{\loga{X},x}$ and  $M^{\sharp}_{\loga{Z},z}\to M^{\sharp}_{\loga{Y},y}$ are injective,
   then the coproduct
   $$M^{\sharp}_{\loga{X},x}\oplus_{M^{\sharp}_{\loga{Z},z}} M^{\sharp}_{\loga{Y},y}$$ is already sharp, by \cite[4.1.12]{GaRa}.

 The log structure on $\loga{W}$ is coherent \cite[2.6]{kato-log}, but in general it is neither integral, nor saturated. The fibered product
 $$\loga{X}\times^{\fs}_{\loga{Z}}\loga{Y}$$ in the category of fine and saturated log schemes can be constructed by
 consecutively applying the functors $(\cdot)^{\mathrm{f}}$ and $(\cdot)^{\mathrm{fs}}$ from \cite[10.2.36(i)]{GaRa} to $\loga{W}$.
 A subtle point of this construction is that it changes the underlying scheme $\mcl{W}$.
  If $\loga{W}\to \Spec^{\dagger}\Z[N]$ is a chart for $\loga{W}$, with $N$ a finitely generated monoid, then
 the underlying scheme of $\loga{X}\times^{\fs}_{\loga{Z}}\loga{Y}$ is given by
 $$\mcl{W}\times_{\Z[N]}\Z[N^{\sat}].$$
  The natural morphism $$\loga{X}\times^{\fs}_{\loga{Z}}\loga{Y}\to \loga{W}$$ is a finite morphism on the underlying schemes \cite[10.2.36(ii)]{GaRa}, and it is an isomorphism over the open subscheme of $\loga{W}$ where the log structure is trivial.
 For every point  $w'$ lying above $w$, we have a canonical isomorphism
  $$M^{\sharp}_{\loga{X}\times^{\fs}_{\loga{Z}}\loga{Y},w'}\cong ((M^{\sharp}_{\loga{X},x}\oplus_{M^{\sharp}_{\loga{Z},z}} M^{\sharp}_{\loga{Y},y})^{\sat})^{\sharp}.$$
  This is proven in \cite[2.1.1]{nakayama} for log schemes with respect to the \'etale topology, but the proof is also valid for the Zariski topology.

  The most important class of fs fibered products for our purposes is described in the following proposition.
  \begin{prop}\label{prop:fsfib}
  Let $R$ be a complete discrete valuation ring with quotient field $K$. Let $K'$ be a finite extension of $K$ and denote by $R'$ the integral closure of $R$ in $K'$.
  We denote by $S^{\dagger}$ the scheme $S=\Spec R$ endowed with its standard log structure (see Example \ref{exam:standard}).
  The log scheme $(S')^{\dagger}$ is defined analogously, replacing $R$ by $R'$.

  Let $\loga{X}$ be a fine and saturated log scheme, and let $\loga{X}\to S^+$ be a smooth morphism of log schemes.
  Then the underlying scheme $\mcl{Y}$ of $\loga{Y}=\loga{X}\times^{\fs}_{S^{\dagger}}(S')^{\dagger}$ is the normalization of $\mcl{X}\times_S S'$.
  \end{prop}
  \begin{proof}
    Smoothness is preserved by fine and saturated base change, so that the log scheme $\loga{Y}$ is smooth over $(S')^{\dagger}$. Since $S^{\dagger}$ and $(S')^{\dagger}$ are regular, the log schemes $\loga{X}$ and $\loga{Y}$ are regular, as well, by Proposition \ref{prop:regsm}. Thus $\mcl{X}$ and $\mcl{Y}$ are normal.
    It also follows from Proposition \ref{prop:regsm} that $\mcl{X}$ is flat over $S$ and $\mcl{Y}$ is flat over $S'$.
  The morphism $\mcl{Y}\to \mcl{X}\times_S S'$ is an isomorphism on the generic fibers, because the log structures on $S$ and $S'$ are trivial at their generic points, so that
  $\loga{X}\times_{S^{\dagger}}\Spec(K')$ is already saturated.
 Thus $\mcl{Y}\to \mcl{X}\times_S S'$ is birational; since it is also finite and $\mcl{Y}$ is normal, it is a normalization morphism.
  \end{proof}

\section{Smooth log schemes over discrete valuation rings}
\subsection{Log modifications and ramified base change}\label{sec:stratram}
  Let $R$ be a complete discrete valuation ring with residue field $k$ and quotient field $K$.   We write $S^\dagger$ for the scheme $S=\Spec R$ endowed with its standard log structure (see Example \ref{exam:standard}).
   We fix a uniformizer $t$ in $R$. For every positive integer $n$ we denote by $R(n)$ the extension $R[u]/(u^{n}-t)$ of $R$.
  We write $S(n)^\dagger$ for the scheme $S(n)=\Spec R(n)$ with its standard log structure. The morphism of monoids
  $$((1/n)\N,+)\to (R[u]/(u^n-t),\times)$$ that sends $1/n$ to $u$ induces a chart
  $$S(n)^{\dagger}\to \Spec^{\dagger}\Z[(1/n)\N]$$
  that we call the standard chart for $S(n)^{\dagger}$.
     Whenever $m$ is a positive multiple of $n$, we have a morphism of log schemes $S(m)^{\dagger}\to S(n)^{\dagger}$ associated with the morphism
   of $R$-algebras $$R[u]/(u^n-t)\to R[v]/(v^m-t)\colon u\mapsto v^{m/n}.$$
 The standard charts for $S(m)^{\dagger}$ and $S(n)^{\dagger}$ fit into a chart
 $$\begin{CD}
S(m)^{\dagger} @>>> \Spec^{\dagger}\Z[(1/m)\N]
\\ @VVV @VVV
\\ S(n)^{\dagger} @>>> \Spec^{\dagger}\Z[(1/n)\N]
\end{CD}$$
for the morphism $S(m)^{\dagger}\to S(n)^{\dagger}$, where the morphism of monoids $(1/n)\N\to (1/m)\N$ is the inclusion map.

Let $\loga{X}$ be a smooth fine and saturated log scheme of finite type  over $S^\dagger$.
 Then $\loga{X}$ is regular, by Proposition \ref{prop:regsm}, so that we can apply the constructions from Section \ref{sec:fans} to $\loga{X}$. Proposition \ref{prop:regsm} also implies that
 the underlying scheme $\mcl{X}$ is flat over $S$. We denote by $e_t$ the image of the uniformizer $t$ in the monoid $M_{\loga{X}}(\mcl{X})$.
 Let $x$ be a point in $\mcl{X}_k$.
 Then the structural morphism $\loga{X}\to S^{\dagger}$ induces a local morphism of monoids
 $$\varphi:\mbb{N}\to M^{\sharp}_{\loga{X},x}$$
 that sends $1$ to the image of $e_t$ in $M^{\sharp}_{\loga{X},x}$, which we will still denote by $e_t$.
 We define the {\em root index} $\rho(x)$ to be root index of this morphism $\varphi$.

 Now let $\tau$ be a point of $F(\loga{X})\cap \mcl{X}_k$. Then $M^{\sharp}_{\loga{X},\tau}=M_{F(\loga{X}),\tau}$.
 We set $\rho=\rho(\tau)$, and we
  denote by $\loga{Y}$ the fibered product $$\loga{X}\times^{\fs}_{S^\dagger}S(\rho)^{\dagger}$$ in the category of fine and saturated log schemes (see Section \ref{sec:fsfib}).
   The log scheme $\loga{Y}$ is smooth over $S(\rho)^{\dagger}$ because smoothness is preserved by fs base change.
   The underlying scheme $\mcl{Y}$ is the normalization of $\mcl{X}\times_S S(\rho)$, by Proposition \ref{prop:fsfib}.
    We set $$\widetilde{E}(\tau)^o=\left(\mcl{Y}\times_{\mcl{X}}E(\tau)^o\right)_{\red}.$$
    This is
     a union of logarithmic strata of $\loga{Y}$, each of which has characteristic monoid
    $$(M_{F(\loga{X}),\tau}\oplus_\mbb{N}\frac{1}{\rho}\mbb{N})^{\sat,\sharp}.$$
 By Proposition \ref{prop:root}\eqref{it:rootdiv}, we know that the natural morphism
 $$M_{F(\loga{X}),\tau}\to (M_{F(\loga{X}),\tau}\oplus_\mbb{N}\frac{1}{\rho}\mbb{N})^{\sat,\sharp}$$ is an isomorphism. If $\tau$ is a point of $F(\loga{X})$ that is not contained in  $\mcl{X}_k$, then we set $\widetilde{E}(\tau)^o=E(\tau)^o$. 

\begin{eg}\label{exam:torsor}
 Let $\loga{X}$ be a smooth fs log scheme of finite type over $S^{\dagger}$.
 Assume that the underlying scheme $\mcl{X}$ is regular; then $\mcl{X}_k$ is a divisor with strict normal crossings (see Example \ref{ex:snc}).  We write
$$\mcl{X}_k=\sum_{i\in I}N_i E_i$$ where $E_i,\,i\in I$ are the prime components of $\mcl{X}_k$ and the coefficients $N_i$ are their multiplicities.
Let $x$ be a point of $\mcl{X}_k$ and let $J$ be the set of indices $j\in I$ such that $x$ lies on $E_j$. Then there exists an isomorphism of monoids
$$\psi\colon M^{\sharp}_{\loga{X},x}\to \N^{J}\times \N^{h},$$ for some integer $h\geq 0$, such that $\psi(e_t)=((N_j)_{j\in J},0)$. Here $h$ is the number of irreducible components of the boundary of $\loga{X}$ that pass through $x$ and that are horizontal, that is, not contained in the special fiber $\mcl{X}_k$.
 The root index $\rho(x)$ is the greatest common divisor of the
multiplicities $N_j,\,j\in J$.

 If $\tau$ is a point of $F(\loga{X})\cap \mcl{X}_k$ and $\rho(\tau)$ is not divisible by the characteristic of $k$, then
 $\widetilde{E}(\tau)^o\to E(\tau)^o$ has a canonical structure of a $\mu_{\rho(\tau)}$-torsor, which is described explicitly in \cite[\S2.3]{Ni-tameram}.
\end{eg}

The following results constitute a key step in the calculation of the motivic zeta function.

\begin{lem}\label{lemm:root}
Let $\loga{X}$ be a smooth fs log scheme of finite type over $S^\dagger$ and let $\tau$ be a point of $F(\loga{X})\cap \mcl{X}_k$ of root index $\rho=\rho(\tau)$.
 Let $m$ be any positive multiple of $\rho$
and  denote by $\loga{Z}$ be the fibered product $$\loga{X}\times^{\fs}_{S^\dagger}S(m)^{\dagger}$$ in the category of fine and saturated log schemes.
Then the natural morphism
$$(\mcl{Z}\times_{\mcl{X}}E(\tau)^o)_{\red}\to \widetilde{E}(\tau)^o$$
is an isomorphism.
\end{lem}
\begin{proof}
We set $$\loga{Y}=\loga{X}\times^{\fs}_{S^\dagger}S(\rho)^{\dagger},$$
as before. We have already recalled that the log scheme $\loga{Y}$ is regular and that $\widetilde{E}(\tau)^o$ is a union of logarithmic strata with characteristic monoid
    $$M=(M_{F(\loga{X}),\tau}\oplus_\mbb{N}\frac{1}{\rho}\mbb{N})^{\sat,\sharp}\cong M_{F(\loga{X}),\tau}.$$ The monoid $M$ is endowed with a local morphism $\varphi:(1/\rho)\N\to M$ induced by $\loga{Y}\to S(\rho)^{\dagger}$, and the morphism $\varphi$ has root index $1$ by Proposition \ref{prop:root}\eqref{it:rootdiv}.

 Let $y$ be a point of $\widetilde{E}(\tau)^o\subset \mcl{Y}$.
By Section \ref{sec:charts}, locally around $y$,
we can find a chart
       for the morphism $\loga{Y}\to S(\rho)^{\dagger}$ of the form
       $$\begin{CD}
      \loga{Y}@>>> \Spec^{\dagger} \Z[M\times \Z]
       \\ @VVV @VVV
     \\  S(\rho)^{\dagger}@>>> \Spec^{\dagger} \Z[\frac{1}{\rho}\N]
       \end{CD}$$
       where the lower horizontal morphism is the standard chart for $S(\rho)^{\dagger}$, and
       $(1/\rho)\N\to M\times \Z$ is a morphism of monoids such that the composition with the projection $M\times \Z\to M$ coincides with $\varphi$.
      We set
      $$N=(M\times \Z)\oplus_{\frac{1}{\rho}\N}\frac{1}{m}\N.$$
      As we have recalled in Section \ref{sec:fsfib}, the $fs$ fibered product $\loga{Z}$ is obtained by first taking the fibered product $$\loga{W}=\loga{Y}\times_{S(\rho)^{\dagger}}S(m)^{\dagger}$$ in the category of
      log schemes; the underlying scheme of $\loga{W}$ is the fibered product $\mcl{Y}\times_{S(\rho)}S(m)$, and,  locally around $y$,
      our chart for $\loga{Y}\to S(\rho)^{\dagger}$ induces a chart $\loga{W}\to \Spec^{\dagger} \Z[N]$.
       Then, over some open neighbourhood of $y$ in $\mcl{Y}$,
  the underlying scheme of $\loga{Z}$ is given by
       $$\mcl{Z}= \left(\mcl{Y}\times_{S(\rho)}S(m)\right)\times_{\Z[N]}\Z[N^{\sat}],$$
        and the morphism $\loga{Z}\to \Spec^{\dagger} \Z[N^{\sat}]$ is a chart for the log structure on $\loga{Z}$.

       By Section \ref{sec:charts}, all the elements in the image of the morphism $M^+\times \Z\to \mathcal{O}_{\mcl{Y},y}$ vanish in $\mathcal{O}_{\widetilde{E}(\tau)^o,y}$. By construction, $(m/\rho)N^+$ is contained in $M^+\times \Z$. Thus for every element $h$ in the image of $N^+\to \mathcal{O}_{\mcl{W},y}$, we have that $h^{m/\rho}$ vanishes in $\mathcal{O}_{\widetilde{E}(\tau)^o,y}$. Since $\mathcal{O}_{\widetilde{E}(\tau)^o,y}$ is reduced, we conclude that
       the morphism $\Z[N]\to \mathcal{O}_{\widetilde{E}(\tau)^o,y}$ factors through
       $\Z[N]/\langle N^+\rangle$, where $\langle N^+\rangle$ denotes the ideal generated by $N^+$.

It is obvious that the
   morphism $\Z[N^{\times}]\to \Z[N]$ induces an isomorphism
   $$\Z[N^{\times}]\to \Z[N]/\langle N^+\rangle.$$
   We have a similar isomorphism $$\Z[(N^{\sat})^{\times}]\to \Z[N^{\sat}]/\langle (N^{\sat})^+\rangle$$ for the monoid $N^{\sat}$, and the morphism $\Z[N^{\times}]\to \Z[(N^{\sat})^{\times}]$ is an isomorphism, by Proposition \ref{prop:root}\eqref{it:rootsharp}.
   Moreover, $\langle(N^{\sat})^+\rangle$ is the radical of the ideal in  $\Z[N^{\sat}]$ generated by $N^+$.
          Therefore, over an open neighbourhood of $y$, the $\widetilde{E}(\tau)^o$-scheme $(\mcl{Z}\times_{\mcl{X}}E(\tau)^o)_{\red}$  is isomorphic to
          $$\left(\widetilde{E}(\tau)^o\times_{\Z[N]}\Z[N^{\sat}]\right)_{\red}\cong \widetilde{E}(\tau)^o\times_{\Z[N^{\times}]}\Z[(N^{\sat})^{\times}]\cong \widetilde{E}(\tau)^o.$$
     \end{proof}

\begin{prop}\label{prop:logblup}
Let $\loga{X}$ be a smooth fs log scheme of finite type over $S^\dagger$.
 Let $\psi:F'\to F$ be a fine and saturated proper subdivision of $F=F(\loga{X})$, and denote by $h:\loga{(X')}\to \loga{X}$ the corresponding morphism of log schemes.
 Let $\tau'$ be a point of $F'$ and set $\tau=\psi(\tau')$. Then there exists a natural morphism of $E(\tau)^o$-schemes
 $$\widetilde{E}(\tau')^o\to \widetilde{E}(\tau)^o$$ such that $\widetilde{E}(\tau')^o$ is a $\mathbb{G}^{r(\tau)-r(\tau')}_{m,k}$-torsor over $\widetilde{E}(\tau)^o$.
\end{prop}
\begin{proof}
Let $m$ be a positive integer that is divisible by both $\rho(\tau)$ and $\rho(\tau')$ and set
$$\loga{Z}=\loga{X}\times^{\fs}_{S^\dagger}S(m)^{\dagger},\quad \loga{(Z')}=\loga{(X')}\times^{\fs}_{S^\dagger}S(m)^{\dagger}.$$
By Lemma \ref{lemm:root}, the morphism $h_m:\loga{(Z')}\to \loga{Z}$ obtained from $h$ by fs base change induces a morphism of $E(\tau)^o$-schemes $\widetilde{E}(\tau')^o\to \widetilde{E}(\tau)^o$.

 We will prove that the morphism $h$ induced by the subdivision $\psi$ is compatible with fs base change, in the following sense.
   The refinement $\psi:F'\to F$ induces a refinement $$\psi_m:F(\loga{Z})\times^{\fs}_{F}F' \to F(\loga{Z})$$
 where the fibered product is taken in the category of fine and saturated fans.
  We claim that the morphism of log schemes induced by this refinement is precisely the morphism $h_m:\loga{(Z')}\to \loga{Z}$.
  Since $\widetilde{E}(\tau)^o$ is a union of logarithmic strata of rank $r(\tau)$ in $\loga{Z}$ and $\widetilde{E}(\tau')^o$ is a union of logarithmic strata of rank $r(\tau')$ in $\loga{(Z')}$, this implies that $\widetilde{E}(\tau')^o$ is a $\mathbb{G}^{r(\tau)-r(\tau')}_{m,k}$-torsor over $\widetilde{E}(\tau)^o$ (see Section \ref{sec:subdiv}).

 It remains to prove that $h_m$ is indeed the morphism induced by the refinement $\psi_m$. The morphism induced by $\psi_m$ is characterized by the following universal property \cite[9.9]{kato}: it is a final object in the category of logarithmic schemes $\loga{W}$ endowed with a morphism $\loga{W}\to \loga{Z}$ and a morphism of monoidal spaces
 $\pi':\loga{W}\to F(\loga{Z})\times^{\fs}_{F}F'$ such that the diagram
 $$\begin{CD}
 \loga{W}@>>> \loga{Z}
 \\ @V\pi' VV @VV\pi V
 \\ F(\loga{Z})\times^{\fs}_{F}F' @>>\psi_m > F(\loga{Z})
 \end{CD}$$
 commutes. If $\loga{W}$ is such a final object, then we have a canonical morphism $\loga{(Z')}\to \loga{W}$ of log schemes over $\loga{Z}$.
 Conversely, applying the universal properties for the morphism $\loga{(X')}\to \loga{X}$ and the fs base change to $S(m)^{\dagger}$, we obtain a morphism $\loga{W}\to \loga{(Z')}$ of log schemes over $\loga{Z}$. These two morphisms are mutually inverse, so that $\loga{W}$ is isomorphic to $\loga{(Z')}$ over $\loga{Z}$.
\end{proof}

\section{Motivic zeta functions}
We denote by $R$ a complete discrete valuation ring with residue field $k$ and quotient field $K$. We assume that $k$ is perfect and we fix a uniformizer $t$ in $R$.
 For every positive integer $n$, we write $R(n)=R[u]/(u^n-t)$ and $K(n)=K[u]/(u^n-t)$. We write $S^\dagger$ and $S(n)^\dagger$ for the schemes $S=\Spec R$ and $S(n)=\Spec R(n)$ endowed with their standard log structures.
\subsection{Grothendieck rings and geometric series}\label{sec:groth}
 If $R$ has equal characteristic, then for every noetherian $k$-scheme $X$, we denote by  $\mathcal{M}_{X}$ the Grothendieck ring of varieties over $X$, localized with respect to the class $\LL$ of the affine line $\mbb{A}^1_{X}$.
 If $R$ has mixed characteristic, we use the same notation, but we replace the Grothendieck ring of varieties by its {\em modified} version, which means that we identify the classes of universally homeomorphic $X$-schemes of finite type -- see \cite[\S3.8]{NS-K0}.
 In the calculation of the motivic zeta function, we will need to consider some specific geometric series in $\LL^{-1}$. The standard technique is to pass to the completion $\widehat{\mathcal{M}}_X$ of $\mathcal{M}_X$ with respect to the dimensional filtration. However, since it is not known whether the completion morphism $\mathcal{M}_X\to \widehat{\mathcal{M}}_X$ is injective,
 we will use a different method to avoid any loss of information. We start with an elementary lemma.

 \begin{lem}\label{lemm:cone}
Let $M$ be a sharp, fine and saturated monoid of dimension $d$.
 For every morphism of monoids $u:M\to \N$, we denote by $u^{\gp}:M^{\gp}\to \Z$ the induced morphism of groups.
Let $m$ be an element of $M^{\gp}$ and let $n$ be an element of $M\setminus \{0\}$.
 Let $u_1,\ldots,u_r$ be the generators of the one-dimensional faces of $M^{\vee}$, and denote by $I$ the set
 of indices $i$ in $\{1,\ldots,r\}$ such that $u_i(n)>0$. We assume that $u_j^{\gp}(m)=1$ for every $j\notin I$.
   Then the series
\begin{equation}\label{eq:series}
(L-1)^{d-1}\sum_{u \in M^{\vee,\loc} } L^{-u^{\gp}(m)}  T^{u(n)}
\end{equation}
in the variables $L$ and $T$ lies in the subring
$$\Z[L,L^{-1},T]\left[\frac{T}{1-L^{-u^{\gp}_i(m)}T^{u_i(n)}}\right]_{i\in I}$$
of $\Z\llb L^{-1},T \rrb [L]$.
\end{lem}
\begin{proof}
A sharp, fine and saturated monoid is called {\em simplicial} if
 its number of one-dimensional faces is equal to its dimension.
 We can subdivide $M^{\vee}$ into a fan of simplicial monoids without inserting new one-dimensional faces, and such a subdivision
 gives rise to a partition of $M^{\vee,\loc}$. Thus we may assume from the start that $M^{\vee}$ is simplicial, so that $d=r$ and
 $u_1,\ldots,u_r$ form a basis for the $\Q$-vector space $(M^{\vee})^{\gp}\otimes_{\Z}\Q$.
  Denote by $P$ the fundamental parallelepiped
  $$P=\{\lambda_1u_1+\ldots+\lambda_r u_r\in M^{\vee,\loc}\,|\,\lambda_i\in \Q\cap (0,1]\,\}$$
  in $M^{\vee,\loc}$. Then $P$ is a finite set, and we have
  $$\eqref{eq:series}=(L-1)^{r-1}\left(\sum_{u\in P}L^{-u^{\gp}(m)}  T^{u(n)}\right)\prod_{i=1}^r \frac{1}{1-L^{-u^{\gp}_i(m)}T^{u_i(n)}}.$$
  Now the result follows from the assumption that for every $i$, either $u_i(n)>0$ or $u^{\gp}_i(m)=1$; note that at most $r-1$ of the values $u_i(n)$ vanish,  because $n\neq 0$.
\end{proof}

 Keeping the notations and assumptions of Lemma \ref{lemm:cone}, we {\em define}
 $$(\LL-1)^{d-1}\sum_{u \in M^{\vee,\loc} } \LL^{-u^{\gp}(m)}  T^{u(n)}$$
as the value of $$(L-1)^{d-1}\sum_{u \in M^{\vee,\loc} } L^{-u^{\gp}(m)}  T^{u(n)}$$ at $L=\LL$. Lemma \ref{lemm:cone} guarantees that this is a well-defined element
of $\mathcal{M}_k\llb T \rrb$.

\subsection{Definition of the motivic zeta function}\label{sec:defzeta}
Let $\mcl{X}$ be an $R$-scheme of finite type with smooth generic fiber $\mcl{X}_K$, and let $\omega$ be a volume form on $\mcl{X}_K$ (that is, a nowhere vanishing differential form of maximal degree on each connected component of $\mcl{X}_K$).
 A {\em N\'eron smoothening} of $\mcl{X}$ is a morphism of finite type $h:\mcl{Y}\to \mcl{X}$ such that $\mcl{Y}$ is smooth over $R$, $h_K:\mcl{Y}_K\to \mcl{X}_K$ is an isomorphism, and
 the natural map $\mcl{Y}(R')\to \mcl{X}(R')$ is a bijection for every finite unramified extension $R'$ of $R$. Such a N\'eron smoothening always exists, by \cite[3.1.3]{BLR}.
 For every connected component $C$ of $\mcl{Y}_k$, we denote by $\ord_C\omega$ the unique integer $a$ such that $t^{-a}\omega$ extends to a relative volume form on $\mcl{Y}$ locally around the generic point of $C$.

\begin{defi}[Loeser-Sebag]\label{def:motint}
The motivic integral of $\omega$ on $\mcl{X}$ is defined by
$$\int_{\mcl{X}}|\omega|=\sum_{C\in \pi_0(\mcl{Y}_k)}[C]\LL^{-\ord_C\omega}\quad \in \mathcal{M}_{\mcl{X}_k}$$
where $\mcl{Y}\to \mcl{X}$ is any N\'eron smoothening and $\pi_0(\mcl{Y}_k)$ is the set of connected components of $\mcl{Y}_k$.
\end{defi}
 It is a deep fact that this definition does not depend on the choice of a N\'eron smoothening; the proof relies on the theory of motivic integration \cite{motrigid}.
   Definition \ref{def:motint} can be interpreted as a motivic upgrade of the integral of a volume form on a compact $p$-adic manifold \cite[\S4.6]{motrigid}.

 The motivic zeta function of the pair $(\mcl{X},\omega)$ is a generating series that measures how the motivic integral in Definition \ref{def:motint} changes under ramified extensions of $R$. For every positive integer $n$, we set $\mcl{X}(n)=\mcl{X}\times_R R(n)$, and we denote by $\omega(n)$ the pullback of $\omega$ to the generic fiber of
 $\mcl{X}(n)$.

 \begin{defi}\label{def:motzeta}
The motivic zeta function of the pair $(\mcl{X},\omega)$ is the generating series
$$Z_{\mcl{X},\omega}(T)=\sum_{n>0}\left(\int_{\mcl{X}(n)}|\omega(n)|\right)T^n\quad \in \mathcal{M}_{\mcl{X}_k}\llb T \rrb.$$
\end{defi}

 Beware that this definition depends on the choice of the uniformizer $t$, except when $k$ has characteristic zero and contains all the roots of unity: in that case,
 $K(n)$ is the unique degree $n$ extension of $K$, up to $K$-isomorphism.

If $h:\mcl{X}'\to \mcl{X}$ is a proper morphism of $R$-schemes such that $h_K:\mcl{X}'_K\to \mcl{X}_K$ is an isomorphism, then it follows immediately from the definition that we can recover $Z_{\mcl{X},\omega}(T)$ from $Z_{\mcl{X}',\omega}(T)$ by specializing the coefficients with respect to the forgetful group homomorphism
$\mathcal{M}_{\mcl{X}'_k}\to \mathcal{M}_{\mcl{X}_k}$. Thus we can compute $Z_{\mcl{X},\omega}(T)$ after a suitable proper modification of $\mcl{X}$.
The principal aim of this paper is to establish an explicit formula for $Z_{\mcl{X},\omega}(T)$ in the case where $\mcl{X}$ is smooth over $S^\dagger$ with respect to a suitable choice of log structure on $\mcl{X}$.

\subsection{Explicit formula on a log smooth model}
Let $\loga{X}$ be a smooth fs log scheme of finite type over $S^\dagger$, and denote by $D$ its reduced boundary divisor, which was defined in Section \ref{sec:fans}.  We write $F=F(\loga{X})$ for the fan associated with $\loga{X}$, and we denote by $e_t$ the image of the uniformizer $t$ in the monoid of global sections of $M_{\loga{X}}$.
 We write $F_k$ for the set $F\cap \mcl{X}_k$. This is a finite set, consisting of
 the points in the special fiber $\mcl{X}_k$ whose Zariski closure is a connected component of an intersection of irreducible components of $D$ (this follows from the description of the logarithmic stratification in Section \ref{sec:fans}).

Let $\omega$ be a differential form of maximal degree on $\mcl{X}_K$ that is nowhere vanishing on $\mcl{X}_K\setminus D$.
 Then we can view $\omega$ as a rational section of the relative canonical bundle $\omega_{\loga{X}/S^\dagger}$. As such, it defines a Cartier divisor on
 $\mcl{X}$, which we denote by $\mathrm{div}_{\loga{X}}(\omega)$. This divisor is supported on $D$.
 Let $\tau$ be a point of $F$.
  For every element $u$ of $M_{F,\tau}^{\vee,\loc}$, we set $u(\omega)=u^{\gp}(\overline{f})\in \Z$, where $\overline{f}$ is the residue class
   in $M^{\gp}_{F,\tau}$ of any element $f\in M^{\gp}_{\loga{X},\tau}$ such that  $\mathrm{div}(f)=\mathrm{div}_{\loga{X}}(\omega)$ locally at $\tau$.
   This definition does not depend on the choice of $f$.
   Note that $u(\omega)>0$ if $\tau$ is not contained in $F_k$, because $\mathrm{div}_{\loga{X}}(\omega)\geq D$ on $\mcl{X}_K$.

\begin{thm}\label{thm:main}
Let $\loga{X}$ be a smooth fs log scheme of finite type over $S^\dagger$. We assume that the generic fiber $\mcl{X}_K$ is smooth over $K$ (but we allow the log structure on $\loga{X}$ to be nontrivial on $\mcl{X}_K$). Let $\omega$ be a volume form on $\mcl{X}_K$.
 Then for every $\tau$ in $F_k$, the expression
 \begin{equation}\label{eq:geometric}
 (\LL-1)^{r(\tau)-1}
\sum_{u \in M_{F,\tau}^{\vee,\loc} } \LL^{-u(\omega)}  T^{u(e_t)}
\end{equation}
 is well-defined in $\mathcal{M}_{\mcl{X}_k}\llb T \rrb$, and
  the motivic zeta function of $(\mcl{X},\omega)$ is given by
\begin{equation}\label{eq:motzeta}
Z_{\mcl{X},\omega}(T)=\sum_{\tau\in F_k}  [\wtl{E}(\tau)^o] (\LL-1)^{r(\tau)-1}
\sum_{u \in M_{F,\tau}^{\vee,\loc} } \LL^{-u(\omega)}  T^{u(e_t)} \end{equation}
in $\mathcal{M}_{\mcl{X}_k}\llb T \rrb$.
\end{thm}
\begin{proof}
We break up the  proof into four steps.

\medskip
{\em Step 1: the expression \eqref{eq:geometric} is well-defined.}
Since $\omega$ is a volume form on $\mcl{X}_K$, the horizontal part of the divisor $\mathrm{div}_{\loga{X}}(\omega)$ coincides with the horizontal part of the reduced boundary divisor $D$ of $\loga{X}$. This means that  $u(\omega)=1$ for every $\tau\in F_k$ and every generator $u$ of a one-dimensional face of $M_{F,\tau}$ such that $u(e_t)=0$.
 Hence, Lemma \ref{lemm:cone} guarantees that \eqref{eq:geometric} is a well-defined element of $\mathcal{M}_{\mcl{X}_k}\llb T \rrb$.

\medskip
{\em Step 2: invariance under log modifications.}
 We will show that the right hand side of \eqref{eq:motzeta} is invariant under the log modification
$h:\loga{(X')}\to \loga{X}$ induced by any fs proper subdivision $\psi:F'\to F$ that is an isomorphism over $F\cap \mcl{X}_K$, or equivalently, such that $h_K:\mcl{X}'_K\to \mcl{X}_K$ is an isomorphism.

 Let $\tau$ be a point in $F_k$. Then, by the definition of a proper subdivision, the morphism $\psi$ induces a bijection between $M_{F,\tau}^{\vee,\loc}$ and
 the disjoint union of the sets $M_{F,\tau'}^{\vee,\loc}$ where $\tau'$ runs through the set of points in $\psi^{-1}(\tau)$.
  Since $h$ is an \'etale morphism of log schemes, the pullback of $\mathrm{div}_{\loga{X}}(\omega)$ to
$\mcl{X}'$ coincides with $\mathrm{div}_{\loga{(X')}}(\omega)$. Thus if $u$ is an element of $M_{F,\tau'}^{\vee,\loc}$ for some $\tau'$ in $\psi^{-1}(\tau)$,
 then the value $u(h^*_K\omega)$ computed on $\loga{(X')}$ coincides with the value $u(\omega)$ computed on $\loga{X}$. The same is obviously true for $u(e_t)$. Moreover, by Proposition \ref{prop:logblup}, we have
 \begin{equation}\label{eq:torus}
 [\wtl{E}(\tau')^o]=(\LL-1)^{r(\tau)-r(\tau')} [\wtl{E}(\tau)^o]
 \end{equation}
 for every point $\tau'$ in $\psi^{-1}(\tau)$. Thus the right hand side of \eqref{eq:motzeta} does not change if we replace $\loga{X}$ by $\loga{(X')}$.

 As a side remark, we observe that our assumption that $h_K$ is an isomorphism has only been used to ensure that $h_K^*\omega$ is a volume form on $\mcl{X}'_K$, so that the right hand side of \eqref{eq:motzeta} is still well-defined in $\mathcal{M}_{\mcl{X}_k}\llb T \rrb$ if we replace $\loga{X}$ by $\loga{(X')}$. Our proof actually shows that the right hand side of \eqref{eq:motzeta}, viewed as an element of $$\mathcal{M}_{\mcl{X}_k}\llb T \rrb \left[\frac{1}{\LL^i -1} \right]_{i>0},$$ is invariant under
   {\em any} proper subdivision $\psi:F'\to F$.

\medskip
{\em Step 3: compatibility with fs base change.} We will prove that the formula \eqref{eq:motzeta} is compatible with fs base change, in the following sense.
Let $n$ be a positive integer and denote by $F(n)$ the fan of the smooth log scheme $\loga{X}\times^{\fs}_{S^\dagger}S(n)^{\dagger}$ over $S(n)^{\dagger}$. Let $t(n)$ be a uniformizer in $R(n)$.
 Then for every positive integer $i$, the coefficient of $T^i$ in the expression
 $$\sum_{\tau'\in F(n)_k} (\LL-1)^{r(\tau')-1} [\wtl{E}(\tau')^o]
\sum_{u' \in M_{F(n),\tau'}^{\vee,\loc} } \LL^{-u'(\omega(n))}  T^{u'(e_{t(n)})}$$
 is equal to the coefficient of $T^{in}$ in the right hand side of \eqref{eq:motzeta}.

 To see this, we first observe that Lemma \ref{lemm:root} implies that, for every point $\tau$ of $F_k$, the $k$-scheme $\widetilde{E}(\tau)^o$ is isomorphic to the disjoint union
 of the $k$-schemes $\wtl{E}(\tau')^o$ where $\tau'$ runs over the points of $F(n)_k$ that are mapped to $\tau$ under the morphism of fans $F(n)\to F$.
 Moreover, $M_{F(n),\tau'}$ is canonically isomorphic to $$(M_{F,\tau}\oplus_{\N}\frac{1}{n}\N)^{\sat,\sharp},$$ which yields a bijective correspondence
 between the local morphisms $u':M_{F(n),\tau'}\to \N$ such that $u'(e_{t(n)})=i$ and the local morphisms $u:M_{F,\tau}\to \N$ such that $u(e_t)=in$.
  Now it only remains to notice that $u'(\omega(n))=u(\omega)$ because $$\omega_{\loga{X}\times^{\fs}_{S^\dagger}S(n)^{\dagger}/S(n)^{\dagger}}$$ is canonically isomorphic to the pullback of $\omega_{\loga{X}/S^\dagger}$, by the compatibility of relative log differentials with fs base change.


\medskip
{\em Step 4: proof of the formula.} By \cite[4.6.31]{GaRa}, we can find a proper subdivision $\psi\colon F'\to F$ such that, if we denote by $h\colon \loga{(X')}\to \loga{X}$ the associated morphism of log schemes,
the scheme $\mcl{X}$ is regular and the morphism $h_K\colon \mcl{X}'_K\to \mcl{X}_K$ is an isomorphism. Thus, by Step 2, we may assume right away that $\mcl{X}$ itself is regular.
 We write $\Sm(\mcl{X})$ for the $R$-smooth locus of $\mcl{X}$. Then the open immersion $\Sm(\mcl{X})\to \mcl{X}$ is a N\'eron smoothening, by \cite[3.1.2]{BLR} and the subsequent remark. By Step 3 and the definition of the motivic integral, we only need to consider the coefficient of $T$ in the right hand side of \eqref{eq:motzeta}, and prove the equality
 $$\sum_{\tau\in F_k}  [\wtl{E}(\tau)^o] (\LL-1)^{r(\tau)-1}
\sum_{u \in M_{F,\tau}^{\vee,\loc},\,u(e_t)=1 } \LL^{-u(\omega)} =\sum_{C\in \pi_0(\Sm(\mcl{X})_k)}[C]\LL^{-\ord_C\omega}$$
in $\mathcal{M}_{\mcl{X}_k}$.

If $\tau$ is a point in $F_k$, then by Example \ref{exam:torsor}, there exists a local morphism $u:M_{F,\tau}\to \N$ with $u(e_t)=1$ if and only if $\tau$ is contained in
a unique component of $\mcl{X}_k$ and this component has multiplicity $1$ in $\mcl{X}_k$. This is equivalent to the condition that $\tau$ lies in $\Sm(\mcl{X})_k$.
 In that case, the root index of $\tau$ is equal to $1$, so that $\widetilde{E}(\tau)^o=E(\tau)^o$.
By the explicit description of the fan $F$ in Section \ref{sec:fans}, the set  $\Sm(\mcl{X})_k$ is a union of logarithmic strata $E(\tau)^o$.
 Thus for every connected component $C$ of $\Sm(\mcl{X}_k)$,  we can write
 $$[C]=\sum_{\tau\in F\cap C} [E(\tau)^o]$$ in $K_0(\Var_{\mcl{X}_k})$.

  Therefore, it only remains to prove the following property: let $\tau$ be a point in $F\cap \Sm(\mcl{X})_k$ and denote by $C(\tau)$ the unique connected component of $\Sm(\mcl{X})_k$ containing $\tau$. Then we have
\begin{equation}\label{eq:order}
(\LL-1)^{r(\tau)-1}\sum_{u \in M_{F,\tau}^{\vee,\loc},\,u(e_t)=1} \LL^{-u(\omega)}=\LL^{-\ord_{C(\tau)}\omega}
\end{equation}
 in $\mathcal{M}_{k}$ (here we again use Lemma \ref{lemm:cone} to view the left hand side as an element of $\mathcal{M}_{k}$).

 First, we consider the case where the log structure at $\tau$ is vertical (this means that every irreducible component of the boundary divisor $D$ that passes through $\tau$ is contained in the special fiber $\mcl{X}_k$). Then $\tau$ is the generic point of $C(\tau)$, the monoid $M_{F,\tau}$ is isomorphic to $\N$, and $e_t$ is its unique generator.
  Thus the only morphism $u$ contributing to the sum in the left hand side of \eqref{eq:order} is the identity morphism $u:\N\to \N$. Now the equality follows from the fact that
  locally around $\tau$, we have a canonical isomorphism $\omega_{\mcl{X}/S}\cong  \omega_{\loga{X}/S^{\dagger}}$ because the morphism
  $\loga{X}\to S^{\dagger}$ is strict at $\tau$.

 Finally, we generalize the result to the case where the log structure at $\tau$ is not vertical. By Example \ref{exam:torsor}, there exists an isomorphism
 $M_{F,\tau}\to \N\times \N^h$ for some integer $h\geq 0$ such that the morphism $\N\to M_{F,\tau}$ is given by $1\mapsto (1,0)$. In this case,
  restriction to $\N^h$ defines a bijection between the set of local morphisms $u:M_{F,\tau}\to \N$ mapping $e_t$ to $1$ and the set of local morphisms $u':\N^h\to \N$.
 Since $\omega$ is a volume form on $\mcl{X}_K$, we have $u(\omega)=\ord_{C(\tau)}\omega + u'(1,\ldots,1)$. Hence,
\begin{eqnarray*}
(\LL-1)^{r(\tau)-1}\sum_{u \in M_{F,\tau}^{\vee,\loc},\,u(e_t)=1} \LL^{-u(\omega)}&=& \LL^{-\ord_{C(\tau)}\omega}(\LL-1)^{h}\sum_{u' \in (\N^h)^{\vee,\loc}} \LL^{-u'(1,\ldots,1)}
\\ &=&\LL^{-\ord_{C(\tau)}\omega}(\LL-1)^{h}\sum_{i_1,\ldots,i_h>0}\LL^{-(i_1+\ldots +i_h)}
\\&=&\LL^{-\ord_{C(\tau)}\omega}
\end{eqnarray*}
in $\mathcal{M}_k$.
\end{proof}

As a special case of Theorem \ref{thm:main}, we recover a generalization to arbitrary characteristic of the formula for strict normal crossings models from \cite[7.7]{NiSe}. Beware that in \cite{NiSe}, the motivic integrals were renormalized by multiplying them with $\LL^{-d}$, where $d$ is the relative dimension of $\mcl{X}$ over $R$.

\begin{cor}\label{cor:snc}
Let $\mcl{X}$ be a regular flat $R$-scheme of finite type such that $\mcl{X}_k$ is a strict normal crossings divisor,  and write
$$\mcl{X}_k=\sum_{i\in I}N_i E_i.$$
 Denote by $\loga{X}$ the log scheme obtained by endowing $\mcl{X}$ with the divisorial log structure induced by
 $\mcl{X}_k$, and assume that $\loga{X}$ is smooth over $S^{\dagger}$ (this is automatic when $k$ has characteristic zero, by Proposition \ref{prop:regsm}). Let $\omega$ be a volume form on $\mcl{X}_K$.

 For every non-empty subset $J$ of $I$, we set
$$E_J^o=(\bigcap_{j\in J}E_j)\setminus (\bigcup_{i\notin J}E_i)$$ and  $N_J=\gcd\{N_j\,|\,j\in J\}$. We denote by $\widetilde{E}_J^o$ the inverse image of $E_J^o$ in the normalization of $\mcl{X}\times_R R(N_J)$. Let $\nu_i$ be the multiplicity of $E_i$ in the divisor $\mathrm{div}_{\loga{X}}(\omega)$, for every $i$ in $I$. Then the motivic zeta function of $(\mcl{X},\omega)$ is given by
$$Z_{\mcl{X},\omega}(T)=\sum_{\emptyset \neq J\subset I}[\widetilde{E}_J^o](\LL-1)^{|J|-1}\prod_{j\in J}\frac{\LL^{-\nu_j}T^{N_j} }{1-\LL^{-\nu_j}T^{N_j}}$$
in $\mathcal{M}_{\mcl{X}_k}\llb T \rrb $.
\end{cor}
\begin{proof}
 By the explicit description of the logarithmic stratification in Section \ref{sec:fans}, the set $E_J^o$ is the union of the strata
 $E(\tau)^o$ where $\tau$ runs through the intersection $F(\loga{X})\cap E_J^o$.
By Example \ref{exam:torsor}, the scheme $\widetilde{E}(\tau)^o$ is the reduced inverse image of $E(\tau)^o$ in the normalization
 of $\mcl{X}\times_R R(N_J)$. Thus by the scissor relations in the Grothendieck ring, we have
 $$[\widetilde{E}^o_J]=\sum_{\tau\in F(\loga{X})\cap E_J^o}[\widetilde{E}(\tau)^o].$$
 Now the description of the characteristic monoids of $\loga{X}$ in Example \ref{exam:torsor}  shows that the expression for $Z_{\mcl{X},\omega}(T)$ in the statement is a particular case of the formula \eqref{eq:motzeta} in Theorem \ref{thm:main}.
\end{proof}

\subsection{Poles of the motivic zeta function}
Theorem \ref{thm:main} yields interesting information on the poles of the motivic zeta function. Since the localized Grothendieck ring of varieties is not a domain,
 the notion of a pole requires some care; see \cite{RoVe}. To circumvent this issue, we introduce the following definition.

 \begin{defi}\label{def:poles} Let $X$ be a Noetherian $k$-scheme and let
  $Z(T)$ be an element of $\mathcal{M}_X \llb T \rrb$. Let $\mathcal{P}$ be a set of rational numbers. We say that $\mathcal{P}$ is a {\em set of candidate poles} for
  $Z(T)$ if $Z(T)$ belongs to the subring
 $$\mathcal{M}_{X}\left[T,\frac{1}{1-\LL^a T^b} \right]_{(a,b)\in \Z\times \Z_{>0},\,a/b\in \mathcal{P}}$$
 of $\mathcal{M}_{X}\llb T \rrb$.
 \end{defi}
 For any reasonable definition of a pole (in particular, the one in \cite{RoVe}), the set of rational poles is included in every set of candidate poles.

\begin{prop}\label{prop:poles}
Let $\loga{X}$ be a smooth fs log scheme of finite type over $S^\dagger$ such that $\mcl{X}_K$ is smooth over $K$. Let $\omega$ be a volume form on $\mcl{X}_K$.
 Write $\mcl{X}_k=\sum_{i\in I}N_i E_i$ and denote by $\nu_i$ the multiplicity of $E_i$ in $\mathrm{div}_{\loga{X}}(\omega)$, for every $i\in I$.
  Then $$\mathcal{P}(\mcl{X})=\{-\frac{\nu_i}{N_i}\,|\,i\in I\}$$ is a set of candidate poles for $Z_{\mcl{X},\omega}(T)$.
 \end{prop}
\begin{proof}
Set $F=F(\loga{X})$ and let $\tau$ be a point of $F\cap \mcl{X}_k$. In view of Theorem \ref{thm:main}, it suffices to show that
$\mathcal{P}(\mcl{X})$ is a set of candidate poles for
$$\sum_{u \in M_{F,\tau}^{\vee,\loc} } \LL^{-u(\omega)}  T^{u(e_t)}.$$
 As we have explained in Section \ref{sec:val}, the one-dimensional faces of $M^{\vee}_{F,\tau}$ correspond canonically to the irreducible components of the boundary divisor $D$ of $\loga{X}$ passing through $\tau$.
 If $u$ is a generator of a one-dimensional face and $E$ is the corresponding component of $D$, then $u(\omega)$ and $u(e_t)$ are the multiplicities of $E$ in $\mathrm{div}_{\loga{X}}$ and $\mcl{X}_k$, respectively. In particular, if $D$ is not included in $\mcl{X}_k$, then
 $u(e_t)=0$ and $u(\omega)=1$, because $\omega$ is a volume form on $\mcl{X}_K$.
Thus the result follows from Lemma \ref{lemm:cone}.
\end{proof}

 Proposition \ref{prop:poles} tells us that, in order to find a set of candidate poles of $Z_{\mcl{X},\omega}(T)$, it is not necessary to take a log resolution
 of the pair $(\mcl{X},\mcl{X}_k)$, which would introduce many redundant candiate poles. This observation is particularly useful in the context of the monodromy conjecture for motivic zeta functions; see Section \ref{sec:DL}.

\section{Generalizations}
\subsection{Formal schemes}\label{sec-formal}
 The definition of the motivic zeta function (Definition \ref{def:motzeta}) can be generalized to the case where $\mcl{X}$ is a formal scheme
 satisfying a suitable finiteness condition (a so-called {\em special} formal scheme in the sense of \cite{berk}, which is also called a formal scheme formally of finite type in the literature). This generalization is carried out in \cite{Ni}, and it is not difficult to extend our formula from Theorem \ref{thm:main} to this setting. The main reason why we have chosen to work in the category of schemes in this article is the lack of suitable references for the basic properties of logarithmic formal schemes on which the proof of our formula relies. However, the proofs for log schemes carry over easily to the formal case, so that the reader who would want to apply Theorem \ref{thm:main} to formal schemes should have no difficulties in making the necessary verifications.

\subsection{Nisnevich  log structures}\label{sec:etale}
 We will now show how Theorem \ref{thm:main} can be adapted to log schemes in the Nisnevich topology. This allows us to compute motivic zeta functions on a larger class of models with components with ``mild'' self-intersections in the special fiber. This generality is needed, for instance, for the applications to motivic zeta functions of
 Calabi-Yau varieties in \cite{HaNi-CY}.
  We will explain in Example \ref{ex:etale} what is the advantage of the Nisnevich topology over the \'etale topology when computing motivic zeta functions.

 Let $Y$ be a Noetherian scheme. A family of morphisms of schemes
 $$\{u_{\alpha}\colon Y_{\alpha}\to Y\,|\,\alpha\in A\}$$
 is called a Nisnevich cover if each morphism is \'etale and, for every point $y$ in $Y$, there exist an element $\alpha$ in $A$ and a point $y_{\alpha}$ in $Y_{\alpha}$ such that
$u_{\alpha}(y_{\alpha})=y$ and the induced morphism of residue fields $\kappa(y)\to \kappa(y_{\alpha})$ is an isomorphism. By taking for $y$ a generic point of $Y$ and applying Noetherian induction, the definition implies that there exists a finite partition of $Y$ into reduced subschemes $Z$ with the property that there exist an index $\alpha$ in $A$ and a subscheme $Z_{\alpha}$ of $Y_{\alpha}$ such that the restriction of $u_{\alpha}$ to $Z_{\alpha}$ is an isomorphism onto $Z$.

These covering families generate a Grothendieck topology, which is called the Nisnevich topology. Let $M_Y\to \mathcal{O}_Y$ be an fs log structure on $Y$ with respect to the \'etale topology (as in \cite{kato-log}).
 We say that the log structure $M_Y$ is Nisnevich if we can find charts for the log structure $M_Y$ locally in the Nisnevich topology on $Y$.

 Let $\loga{X}$ be a smooth fs Nisnevich log scheme of finite type over $S^\dagger$.  Then the sheaf of monoids $M^{\sharp}_{\loga{X}}$ is constructible on the Nisnevich site of $\mcl{X}$, by the same proof as in \cite[10.2.21]{GaRa}.
 We choose a partition $\mathscr{P}$ of $\mcl{X}_k$ into irreducible locally closed subsets $U$ such that the restriction of $M^{\sharp}_{\loga{X}}$ to the Nisnevich site on $U$ is constant. We denote by $P$ the set consisting of the generic points of all the strata $U$ in $\mathscr{P}$.
  For every point $\tau$ in $P$ we will write $E(\tau)^o$ for the unique stratum in $\mathscr{P}$ containing $\tau$, and we denote by $r(\tau)$ the dimension of the monoid $M^{\sharp}_{\loga{X},\tau}$.
    We define the root index $\rho(\tau)$ and the scheme $\widetilde{E}(\tau)^o$ in exactly the same way as before, and we write $e_t$ for the image
  of $t$ in the monoid of global sections of  $M^{\sharp}_{\loga{X}}$. If $\mcl{X}_K$ is smooth over $K$ and $\omega$ is a volume form
  on $\mcl{X}_K$, then we can also simply copy the definition of the value $u(\omega)$ for every local morphism $u:M^{\sharp}_{\loga{X},\tau}\to \N$.

 \begin{thm}\label{thm:etale}
Let $\loga{X}$ be a smooth fs Nisnevich log scheme of finite type over $S^\dagger$. We assume that the generic fiber $\mcl{X}_K$ is smooth over $K$. Let $\omega$ be a volume form on $\mcl{X}_K$.
 Then the motivic zeta function of $(\mcl{X},\omega)$ is given by
\begin{equation}\label{eq:motzeta-etale}
Z_{\mcl{X},\omega}(T)=\sum_{\substack{\tau\in P}} (\LL-1)^{r(\tau)-1} [\wtl{E}(\tau)^o]
\sum_{u \in (M^{\sharp}_{\loga{X},\tau})^{\vee,\loc} } \LL^{-u(\omega)}  T^{u(e_t)} \end{equation}
in $\mathcal{M}_{\mcl{X}_k}\llb T \rrb$.
\end{thm}
\begin{proof}
One can reduce to the Zariski case by observing that the motivic zeta function $Z_{\mcl{X},\omega}(T)$ is local with respect to the Nisnevich topology, in the following sense:
 if $h:\mcl{U}\to \mcl{X}$ is an \'etale morphism of finite type and $Y$ is a subscheme of $\mcl{X}_k$ such that $Y'=\mcl{U}\times_{\mcl{X}}Y\to Y$ is an isomorphism, then
 $Z_{\mcl{X},\omega}(T)$ and $Z_{\mcl{U},h^*_K\omega}(T)$ have the same image under the base change morphisms $\mathcal{M}_{\mcl{U}_k}\llb T \rrb\to \mathcal{M}_{Y}\llb T \rrb$ and $\mathcal{M}_{\mcl{X}_k}\llb T \rrb\to \mathcal{M}_{Y}\llb T \rrb$, respectively. This is an immediate consequence of the definition of the motivic integral.
  Moreover, if $\{Y_1,\ldots,Y_r\}$ is a finite partition of $\mcl{X}_k$ into subschemes and we denote by $Z_i(T)$ the image of  $Z_{\mcl{X},\omega}(T)$ under the composition
 $$\mathcal{M}_{\mcl{X}_k}\llb T \rrb\to \mathcal{M}_{Y_i}\llb T \rrb\to \mathcal{M}_{\mcl{X}_k}\llb T \rrb$$ (base change followed by the forgetful morphism), then
 $$Z_{\mcl{X},\omega}(T)=Z_1(T)+\ldots+Z_r(T).$$ Thus we can compute $Z_{\mcl{X},\omega}(T)$ on a Nisnevich cover of $\mcl{X}$ where the log structure becomes Zariski in the sense of \cite[2.1.1]{niziol}.
 Since the right hand side of \eqref{eq:motzeta-etale} satisfies the analogous localization property with respect to the Zariski topology, the result now follows from the Zariski case that was proven in Theorem \ref{thm:main}.
\end{proof}

\begin{cor}\label{cor:etpoles}
Let $\loga{X}$ be a smooth fs log scheme of finite type over $S^\dagger$ (with respect to the Nisnevich topology) such that $\mcl{X}_K$ is smooth over $K$.
  Let $\omega$ be a volume form on $\mcl{X}_K$.
 Write $\mcl{X}_k=\sum_{i\in I}N_i E_i$ and denote by $\nu_i$ the multiplicity of $E_i$ in $\mathrm{div}_{\loga{X}}(\omega)$, for every $i\in I$.
  Then $$\mathcal{P}(\mcl{X})=\{-\frac{\nu_i}{N_i}\,|\,i\in I\}$$ is a set of candidate poles for $Z_{\mcl{X},\omega}(T)$.
 \end{cor}
 \begin{proof}
The proof is almost identical to the proof of the Zariski case (Proposition \ref{prop:poles}), using the formula in Theorem \ref{thm:etale} instead of Theorem \ref{thm:main}.
 We no longer have a bijective correspondence
 between the generators $u$ of one-dimensional faces of $(M^{\sharp}_{\loga{X},\tau})^{\vee}$ and the irreducible components of the boundary $D$ containing $\tau$, in general,
  because one component may have multiple formal branches at $\tau$ and each of these will give rise to a one-dimensional face of $(M^{\sharp}_{\loga{X},\tau})^{\vee}$.
  However, it remains true that for every generator $u$ of a one-dimensional face of $(M^{\sharp}_{\loga{X},\tau})^{\vee}$, there exists an irreducible component $E$ of $D$ such that $u(e_t)$ equals the multiplicity of $D$ in $\mcl{X}_k$ and $u(\omega)$  equals the multiplicity of $E$ in $\mathrm{div}_{\loga{X}}(\omega)$. This is sufficient to prove the result.
\end{proof}

The following example shows that the formula in Theorem \ref{thm:etale} may fail if we replace the Nisnevich topology by the \'etale topology.

\begin{eg}\label{ex:etale}
Let $R=\R\llb t\rrb$ and set $$\mcl{X}=\Spec R[x,y]/(x^2+y^2-t).$$ We denote by $\loga{X}$ the \'etale log scheme we get by endowing $\mcl{X}$ with the divisorial log structure induced by $\mcl{X}_{\R}$. Then $\loga{X}$ is smooth over $S^+$, since the base change to $R'=\C\llb t \rrb$ is isomorphic to $\Spec R'[u,v]/(uv-t)$.
 However, the log structure is not Nisnevich (that is, the \'etale sheaf $M_{\loga{X}}$ is not the pullback of a sheaf in the Nisnevich topology).

 The \'etale sheaf $M^{\sharp}_{\loga{X}}$ is locally constant on the complement of the origin $O$ of $\mcl{X}_\R$, with geometric stalk $\N$. The geometric stalk
 of $M^{\sharp}_{\loga{X}}$ at $O$ is isomorphic to $\N^2$.
  The line bundle $\omega_{\loga{X}/S^{\dagger}}$ is trivial on $\mcl{X}$ with generator $$\omega=\frac{1}{t}(ydx-xdy).$$
  Blowing up $\mcl{X}$ at $O$, we obtain a regular $R$-scheme whose special fiber has strict normal crossings. Using Corollary \ref{cor:snc}, one computes that the image of $Z_{\mcl{X},\omega}(T)$ under the forgetful morphism $\mathcal{M}_{\mcl{X}_\R}\llb T \rrb\to \mathcal{M}_{\R}\llb T \rrb$
     is equal to $$([C]-[\Spec \C])\frac{ T^2}{1- T^2 }+[\Spec \C](\LL-1)\frac{T}{1-T} +[\Spec \C](\LL-1)\frac{ T^3 }{(1-T)(1-T^2)},$$
 where $C$ is a geometrically connected smooth projective rational curve over $\R$ without rational point.

 The right hand side of \eqref{eq:motzeta-etale} equals
$$[\Spec \C](\LL-1)\frac{T}{1-T}+(\LL-1)\frac{T^2}{(1-T)^2},$$
 which does not agree with our expression for  $Z_{\mcl{X},\omega}(T)$.
 Indeed,
$$[C]-[\Spec \C]\neq \LL-1$$ in $\mathcal{M}_{\R}$, as can be seen by applying the \'etale realization morphism
$$\mathcal{M}_{\R}\to K_0(\Q_\ell[\Z/2\Z]):[X]\mapsto \sum_{i\geq 0}(-1)^i [H^i_{\acute{e}t,c}(X\times_{\R}\C,\Q_\ell)]$$
for any prime $\ell$.

Similar examples can be constructed when $k$ is algebraically closed, for instance by considering
$$\mathcal{X}=\Spec R[x,y,z,z^{-1}]/(x^2+zy^2-t)$$ with $R=k\llb t\rrb$.
\end{eg}

\section{The monodromy action}\label{sec:monodromy}
\subsection{Equivariant Grothendieck rings}
 Let $X$ be a Noetherian scheme and let $G$ be a finite group scheme over $\Z$ that acts on $X$; unless explicitly stated otherwise,
 we will always assume that group schemes act on schemes from the left. Suppose that the action of $G$ on $X$ is {\em good}, which means that $X$ can be covered by $G$-stable affine open subschemes. Then the Grothendieck group $K^G_0(\Var_X)$ of $X$-schemes with $G$-action is the abelian group defined by the following presentation:
 \begin{itemize}
\item {\em Generators:} Isomorphism classes $[Y]$ of $X$-schemes $Y$ of finite type endowed with a good action of $G$ such that the morphism $Y\to X$ is $G$-equivariant; here the isomorphism class is taken with respect to $G$-equivariant isomorphisms.
\item{\em Relations:}
\begin{enumerate}
\item If $Y$ is an $X$-scheme of finite type with good $G$-action and $Z$ is a closed subscheme of $Y$ that is stable under the $G$-action, then
$$[Y]=[Z]+[Y\setminus Z].$$
\item If $Y$ is an $X$-scheme of finite type with good $G$-action and $A\to Y$ is an affine bundle of rank $r$ endowed with an affine lift of the $G$-action on $Y$, then
$$[A]=[\A^r_{\Z}\times_{\Z} Y]$$ where $G$ acts trivially on $\A^r_{\Z}$.
\end{enumerate}
\end{itemize}
We define a ring structure on $K^G_0(\Var_X)$ by means of the multiplication rule $[Y]\cdot [Y']=[Y\times_X Y']$ where $G$ acts diagonally on $Y\times_X Y'$. We write $\LL$ for the class $[\A^1_{\Z}\times_{\Z} X]$ and we set $\mathcal{M}^G_X=K^G_0(\Var_X)[\LL^{-1}]$.

 We will use this definition in the case where $G=\mu_n$, the group scheme of $n$-th roots of unity, for some positive integer $n$.
  If $m$ is a positive multiple of $n$, then the $(m/n)$-th power map $\mu_m\to \mu_n$ induces a ring morphism $\mathcal{M}^{\mu_n}_X\to \mathcal{M}^{\mu_m}_X$.
  We denote by $\widehat{\mu}$ the profinite group scheme of roots of unity and we set
  $$\mathcal{M}^{\widehat{\mu}}_X=\lim_{\stackrel{\longrightarrow}{n>0}}\mathcal{M}^{\mu_n}_{X}$$ where the positive integers $n$ are ordered by the divisibility relation.
   An action of $\widehat{\mu}$ on a Noetherian scheme is called {\em good} if it factors through a good action of $\mu_n$ for some $n>0$. We will need the following elementary result.

   \begin{prop}\label{prop:eqtor}
   Let $Y\to X$ be an equivariant morphism of Noetherian schemes with a good $\mu_n$-action, for some $n>0$. Assume that $Y$ is a $\mathbb{G}^r_{m,\Z}$-torsor
   over $X$  and that the action $$\mathbb{G}^r_{m,\Z}\times_{\Z}Y\to Y$$ is $\mu_n$-equivariant, where $\mu_n$ acts trivially on $\mathbb{G}^r_{m,\Z}$. Then we have $$[Y]=[X](\LL-1)^r$$ in $K^{\mu_n}_0(\Var_X)$.
   \end{prop}
   \begin{proof}
   The torsor $Y$
   can be decomposed as a product $$\mathcal{L}_1^{\ast}\times_X \cdots \times_X \mathcal{L}^{\ast}_{r}$$ where $\mathcal{L}_1,
   \ldots,\mathcal{L}_{r}$ are $\mu_n$-equivariant line bundles on
   $X$ and $\mathcal{L}_i^{\ast}$ is obtained from $\mathcal{L}_i$ by removing the zero section. Now the relations in the equivariant Grothendieck ring immediately imply that
   $$[Y]=[X](\LL-1)^r$$ in $K^{\mu_n}_0(\Var_X)$.
   \end{proof}

\subsection{Monodromy action on the motivic zeta function}
Let $k$ be a field of characteristic zero and set $R=k\llb t\rrb$ and $K=k\llpar t\rrpar$.
Let $\mcl{X}$ be an $R$-scheme of finite type with smooth generic fiber $\mcl{X}_K$, and let $\omega$ be a volume form on $\mcl{X}_K$. Then the definition of the motivic zeta function $Z_{\mcl{X},\omega}(T)$ (Definition \ref{def:motzeta}) can be refined in the following way.
For every positive integer $n$, the finite group scheme $\mu_{n}$ of $n$-th roots of unity acts on $S(n)=\Spec R[u]/(u^n-t)$ from the right {\em via} multiplication on $u$:
 $$R[u]/(u^n-t)\to  R[u]/(u^n-t)\otimes_{\Z} \Z[\zeta]/(\zeta^n-1):u\mapsto \zeta u.$$
  We invert this action to obtain a left action on $S(n)$.
  This induces a left action of $\mu_n$ on $\mcl{X}(n)$.
 One can use this action to
upgrade the motivic integral $$\int_{\mcl{X}(n)}|\omega(n)|$$ to an element in the equivariant Grothendieck ring
 $\mathcal{M}_{\mcl{X}_k}^{\mu_n}$ of $\mcl{X}_k$-varieties with $\mu_n$-action -- see \cite{hartmann}, where one can remove the assumption that $k$ contains all the roots of unity, since it is not needed in the arguments. This equivariant motivic integral can be computed by taking a quasi-projective $\mu_n$-equivariant N\'eron smoothening $\mcl{Y}\to \mcl{X}(n)$ over $R(n)$: then
$$\int_{\mcl{X}(n)}|\omega(n)|=\sum_{i\in \Z}[C(i)]\LL^{-i}\quad \in \mathcal{M}^{\mu_n}_{\mcl{X}_k}$$
 where $C(i)$ is the union of the connected components $C$ of $\mcl{Y}_k$ such that $\ord_{C}\omega(n)=i$; note that $C(i)$ is stable under the action of $\mu_n$, because $\omega(n)$ is defined over $K$. A quasi-projective $\mu_n$-equivariant smoothening $\mcl{Y}\to \mcl{X}(n)$ can always be produced by means of the smoothening algorithm described in the proof of \cite[3.4.2]{BLR}; quasi-projectivity implies that the $\mu_n$-action on $\mcl{Y}$ is good.

 Now we can view the motivic zeta function
 $$Z_{\mcl{X},\omega}(T)=\sum_{n>0}\left(\int_{\mcl{X}(n)}|\omega(n)|\right)T^n$$
 as an object in $\mathcal{M}^{\widehat{\mu}}_{\mcl{X}_k}\llb T \rrb$. We will use the new notation $Z^{\widehat{\mu}}_{\mcl{X},\omega}(T)$ to indicate that we take the $\widehat{\mu}$-action into account.
  On the other hand, the schemes $\widetilde{E}(\tau)^o$ appearing in Theorems \ref{thm:main} and \ref{thm:etale} also carry
 an obvious action of the group scheme $\widehat{\mu}$, because $\mu_n$ acts on the fs base change $$\loga{X}\times_{S^{\dagger}}S(n)^{\dagger}$$ for every $n>0$ {\em via} the left action on $S(n)$.

 \begin{thm}\label{thm:equiv}
 If $k$ has characteristic zero, then Theorems \ref{thm:main} and \ref{thm:etale}
  are valid already for the equivariant motivic zeta function
 $Z^{\widehat{\mu}}_{\mcl{X},\omega}(T)$ in $\mathcal{M}^{\widehat{\mu}}_{\mcl{X}_k}\llb T \rrb$.
 \end{thm}
 \begin{proof}
 We can follow a similar strategy as in the proof of Theorem \ref{thm:main}. Let $k^a$ be an algebraic closure of $k$ and set $(S^a)^{\dagger}=\Spec k^a\llb t\rrb$ with its standard log structure. To compute the degree $n$ coefficient of $Z_{\mcl{X},\omega}(T)$ we can  choose a regular subdivision of the fan $F^a(n)$ of $\mcl{X}(n)^{\dagger}\times_{S^\dagger}(S^a)^{\dagger}$ that is equivariant with respect to the actions of $\mu_n(k^a)$ and the Galois group $\Gal(k^a/k)$. This can be achieved by canonical equivariant resolution of singularities for toroidal embeddings (see for instance the remark on p.~33 of \cite{wang}).
        The induced log modification of $\mcl{X}(n)$ is a regular scheme and its smooth locus is a $\mu_n$-equivariant N\'eron smoothening of $\mcl{X}(n)$. Then a similar computation as in the last step of the proof of Theorem \ref{thm:main} yields the desired result.

  The only step that requires further clarification is the equality \eqref{eq:torus}: we need to show that it remains valid in the equivariant Grothendieck ring $\mathcal{M}^{\widehat{\mu}}_{\mcl{X}_k}$. For every fixed point $\sigma$ of the $\mu_n(k^a)$-action on $F^a(n)$, the group $\mu_n(k^a)$ also acts trivially on the stalk
   of $M_{F^a(n)}$ at $\sigma$ by \cite[2.1.1]{nakayama}. In the notation of \eqref{eq:torus}, this means that the natural morphism $\widetilde{E}(\tau')^o\to \widetilde{E}(\tau)^o$ is a $\mu_n$-equivariant torsor with translation group $\mathbb{G}_{m,k}^{r(\tau)-r(\tau')}$, where $\mu_n$ acts trivially on $\mathbb{G}_{m,k}^{r(\tau)-r(\tau')}$.  Now it follows from Proposition \ref{prop:eqtor} that  $$[\wtl{E}(\tau')^o]=(\LL-1)^{r(\tau)-r(\tau')} [\wtl{E}(\tau)^o]$$ in $\mathcal{M}^{\widehat{\mu}}_{\mcl{X}_k}$.
\end{proof}

The definition of a set of candidate poles (Definition \ref{def:poles}) can be generalized to elements of $\mathcal{M}^{\widehat{\mu}}_{\mcl{X}_k}\llb T \rrb$ in the obvious way.
 Then we can deduce the following result from Theorem \ref{thm:equiv}.

\begin{cor}\label{cor:monopoles}
Assume that $k$ has characteristic zero. Let $\loga{X}$ be a smooth fs log scheme of finite type over $S^\dagger$ (with respect to the Zariski or Nisnevich topology) such that $\mcl{X}_K$ is smooth over $K$. Let $\omega$ be a volume form on $\mcl{X}_K$.
 Write $\mcl{X}_k=\sum_{i\in I}N_i E_i$ and denote by $\nu_i$ the multiplicity of $E_i$ in $\mathrm{div}_{\loga{X}}(\omega)$, for every $i\in I$.
  Then $$\mathcal{P}(\mcl{X})=\{-\frac{\nu_i}{N_i}\,|\,i\in I\}$$ is a set of candidate poles for $Z^{\widehat{\mu}}_{\mcl{X},\omega}(T)$.
 \end{cor}
 \begin{proof}
The argument is entirely similar to the proofs of Proposition \ref{prop:poles} and Corollary \ref{cor:etpoles}.
\end{proof}

\section{Applications to Denef and Loeser's motivic zeta function}\label{sec:DL}
\subsection{The motivic zeta function of Denef-Loeser}
 Let $k$ be a field of characteristic zero, let $X$ be an irreducible smooth $k$-variety and let $$f:X\to \A^1_k=\Spec k[t]$$ be a dominant morphism of $k$-schemes.
  We set $X_0=f^{-1}(0)$. In \cite{DL-barc}, Denef and Loeser have defined the {\em motivic zeta function} $Z_f(T)$ of $f$, which is a power series in $\mathcal{M}^{\widehat{\mu}}_{X_0}\llb T \rrb$ that can be viewed as a motivic upgrade of Igusa's local zeta function for polynomials over $p$-adic fields.
 The famous {\em monodromy conjecture}
 predicts that the set of roots of the Bernstein polynomial of $f$ is a set of candidate poles for $Z_f(T)$ (in the sense of Definition \ref{def:poles}). This has been proven when $X$ has dimension $2$ and for some specific classes of singularities, but the conjecture is wide open in general. In fact, to our best knowledge, the proofs of the dimension $2$ case in the literature consider a slightly weaker conjecture, dealing only with the so-called ``na\"ive'' motivic zeta function, which can be viewed as the quotient of $Z_f(T)$ by the action of $\widehat{\mu}$ (up to multiplication by a factor $\LL-1$). We will explain below how the argument can be refined to prove the conjecture for $Z_f(T)$ (Corollary \ref{cor:twovar}).

  Set $R=k\llb t \rrb$, $K=k\llpar t\rrpar$ and $\mcl{X}=X\times_{k[t]}R$. Let us recall how one can rewrite $Z_f(T)$ as the motivic zeta function of $(\mcl{X},\omega)$ for a suitable volume form $\omega$ on $\mcl{X}_K$.
  Since the definition of the motivic zeta function is local on $X$, we can assume that $X$ carries a volume form $\phi$ over $k$. To this volume form, one can attach a so-called {\em Gelfand-Leray form} $\omega=\phi/df$, which is a volume form on $\mcl{X}_K$ \cite[9.5]{NiSe}. Theorem 9.10 in \cite{NiSe} states that $$Z_{\mcl{X},\omega}(T)=Z_f(\LL T)$$ in $\mathcal{M}_{X_0}\llb T \rrb$ (where we forget the $\widehat{\mu}$-action on the right hand side). This can also be deduced from Corollary \ref{cor:snc} and Denef and Loeser's formula for the motivic zeta function in terms of a log resolution for $f$ \cite[3.3.1]{DL-barc}.
   Using Theorem \ref{thm:equiv}, one can moreover show that this equality holds already for the equivariant motivic zeta function $Z^{\widehat{\mu}}_{\mcl{X},\omega}(T)$ in $\mathcal{M}^{\widehat{\mu}}_{X_0}\llb T \rrb$.
   More precisely, let $h:Y\to X$ be a log resolution for the pair $(X,X_0)$, and write $h^*X_0=\sum_{i\in I}N_i E_i$ and
 $K_{Y/X}=\sum_{i\in I}(\nu_i-1)E_i$. Set $\mcl{Y}=Y\times_{k[t]}R$ and endow it with the divisorial log structure induced by $\mcl{Y}_k$.
  Then the multiplicity of $E_i$ in $\mathrm{div}_{\loga{Y}}(\omega)$ equals $\nu_i-N_i$, so that the expression in Corollary \ref{cor:snc}
   is precisely Denef and Loeser's formula for $Z_f(\LL T)$. Hence, we obtain that
   $$Z^{\widehat{\mu}}_{\mcl{X},\omega}(T)=Z_f(\LL T)$$ in $\mathcal{M}^{\widehat{\mu}}_{X_0}\llb T \rrb$.

  Thus Theorem \ref{thm:equiv} and Corollary \ref{cor:monopoles} also apply to the motivic zeta function of Denef and Loeser. As an illustration, we will apply these results to two particular situations:
 the case where $X$ has dimension $2$, and the case where $f$ is a polynomial that is non-degenerate with respect to its Newton polyhedron. These cases have been studied extensively in the literature; we will explain how some of the main results can be viewed as special cases of Theorem \ref{thm:main}.

 \subsection{The surface case}\label{sec:curves}
 Assume that  $X$ has dimension $2$, and let $h:Y\to X$ be a log resolution for the pair $(X,X_0)$. We write $h^*X_0=\sum_{i\in I}N_i E_i$ and
 $K_{Y/X}=\sum_{i\in I}(\nu_i-1)E_i$. The numbers $N_i$ and $\nu_i$ are called the {\em numerical data} associated with the component $E_i$. It follows from Denef and Loeser's formula \cite[3.3.1]{DL-barc} that
 $$\mathcal{P}=\{-\frac{\nu_i}{N_i}\,|\,i\in I\}$$ is a set of candidate poles for $Z_f(T)$. However, it is known that many of these candidate poles are not actual poles of $Z_f(T)$.   In \cite{veys}, Veys has provided a conceptual explanation for this phenomenon
    by providing a formula for the topological zeta function (a coarser predecessor of the motivic zeta function) in terms of the relative log canonical model of $(X,X_0)$ over $X$. We can now upgrade this result to the motivic zeta function and
     understand it as a special case of Theorem \ref{thm:main}.

     \begin{thm}\label{thm:twovar}
      For every $i\in I$ such that $E_i$ is exceptional with respect to $h$, we denote by $k_i$ the field $H^0(E_i,\mathcal{O}_{E_i})$.
  We define $I_0$ to be the subset of $I$ consisting of the indices $i$ such that $E_i$ is an exceptional component of $h^*X_0$ satisfying $(E_i)^2\geq  -2[k_i:k]$. Then
  $$\mathcal{P}'=\{-\frac{\nu_i}{N_i}\,|\,i\in I\setminus I_0\}$$
  is still a set of candidate poles of $Z_f(T)\in \mathcal{M}^{\widehat{\mu}}_{X_0}\llb T \rrb$.
     \end{thm}
\begin{proof}
     Set $\mcl{Y}=Y\times_{k[t]}R$.
   Contracting all the components $E_i$ with $i\in I_0$ yields a new model $\mcl{Z}$ of $\mcl{X}_K$ that is proper over $\mcl{X}$, namely, the log canonical model of $(\mcl{X},\mcl{X}_k)$ over $\mcl{X}$. We endow $\mcl{Z}$ with the divisorial log structure induced by $\mathcal{Z}_k$. It follows from \cite[\S3]{quotient} that the resulting log scheme $\loga{Z}$ is regular with respect to the \'etale topology, and since $k$ has characteristic zero, this implies that $\loga{Z}$ is smooth over $S^\dagger$ with respect to the \'etale topology (Proposition \ref{prop:regsm}).

   The log structure on $\loga{Z}$ fails to be Zariski precisely at the self-intersection points of components in the strict transform of $X_0$.
   If $k$ is algebraically closed, then the log structure is Nisnevich at these points (because they have algebraically closed residue field) and our result follows immediately from Corollary \ref{cor:monopoles}. For general $k$, we can make the log structure Zariski by blowing up $\loga{Z}$ at each of the self-intersection points (see the proof of \cite[5.4]{niziol}). These blow-ups are log blow-ups, so that the resulting morphism of log schemes $\loga{W}\to\loga{Z}$ is \'etale. Therefore, blowing up at a self-intersection point of a component $E_i$ yields an exceptional divisor with numerical data $N=2N_i$ and $\nu=2\nu_i$. This implies that $\mathcal{P}(\loga{W})=\mathcal{P}'$, so that the result follows from Corollary \ref{cor:monopoles} (applied to the smooth Zariski log scheme $\loga{W}$).
\end{proof}

\begin{cor}\label{cor:twovar}
 There exists a set of candidate poles for $Z_f(T)$ that consists entirely of roots of the Bernstein polynomial of $f$.
 Thus the monodromy conjecture for $Z_f(T)\in \mathcal{M}_{\mcl{X}_k}^{\widehat{\mu}}\llb T\rrb $ holds in dimension $2$.
\end{cor}
\begin{proof}
   Loeser has proven in \cite[III.3.1]{loeser} that every element of $\mathcal{P}'$ is a root of the Bernstein polynomial of $f$.
\end{proof}

  Analogous results have previously appeared in the literature for the $p$-adic zeta function \cite{loeser, strauss}, the topological zeta function \cite{veys} and the so-called ``na{\"\i}ve'' motivic zeta function \cite{rodrigues}.

\subsection{Non-degenerate polynomials}\label{sec:nondeg}
As a second illustration, we will use our results to recover the formula for the motivic zeta function of a  polynomial that is non-degenerate with respect to its Newton polyhedron \cite[\S2.1]{guibert} (see also \cite[\S10]{bories}
for a calculation of the local  ``na{\"i}ve'' motivic zeta function at the origin). In fact, our computations show that the formula in \cite[2.1.3]{guibert} has some flaws; we will explain in Remark \ref{rem:guibert} what needs to be corrected. Let $$f=\sum_{m\in \N^n}a_{m}x^{m}$$ be a non-constant polynomial in $k[x_1,\ldots,x_n]$, where we use the multi-index notation for $x=(x_1,\ldots,x_n)$. We assume that $f(0,\ldots,0)=0$. The {\em support} $S(f)$ of $f$ is the set of $m\in \N^n$ such that $a_{m}$ is non-zero, and the {\em Newton polyhedron} $\Gamma(f)$ of $f$ is the convex hull of $$\bigcup_{m\in S(f)}(m+\R_{\geq 0}^n).$$ For every face $\gamma$ of $\Gamma(f)$, we set
$$f_{\gamma}=\sum_{m \in \gamma\cap \N^n}a_{m}x^{m}.$$
 Then $f$ is called {\em non-degenerate} with respect to its Newton polyhedron if, for every face $\gamma$ of $\Gamma(f)$, the polynomial $f_{\gamma}$ has no critical points in the torus $\mathbb{G}_{m,k}^n$ (this includes the case $\gamma=\Gamma(f)$). This condition was introduced by Kushnirenko in \cite{kouchnirenko}. It guarantees that many interesting invariants of the singularities of $f$ can be computed from the Newton polyhedron in a combinatorial way. In particular, every regular subdivision of the dual fan of $\Gamma(f)$ defines a toric modification of $\A^n_k$ that is a log resolution for the pair $(\A^n_k,\mathrm{div}(f))$. Moreover, if we fix the support $S(f)$, then $f$ is Newton non-degenerate for a generic choice of coefficients $a_{m}$.

 We denote by $\Sigma$ the dual fan of $\Gamma(f)$ and by $h:Y\to \A^n_k$ the toric modification associated with the subdivision $\Sigma$ of $(\R_{\geq 0})^n$.
 We view $Y$ as a $k[t]$-scheme {\em via} the morphism $f\circ h:Y\to \Spec k[t]$ and we set  $\mcl{Y}=Y\times_{k[t]}R$. We denote by $H$ the pullback to $\mcl{Y}$ of the union of the coordinate hyperplanes in $\A^n_k$.
   We endow $\mcl{Y}$ with the divisorial Zariski log structure induced by the divisor
  $\mcl{Y}_k+H$. The result is a Zariski log scheme $\loga{Y}$ over $S^{\dagger}$.

  \begin{prop}\label{prop:nondeg-smooth}
  The log scheme $\loga{Y}$ is fine and saturated, and smooth over $S^{\dagger}$.
  \end{prop}
  \begin{proof}
By Proposition \ref{prop:regsm}, we only need to show that $\loga{Y}$ is regular. Since $f$ has no critical points
on $\mathbb{G}^n_{m,k}$ by the non-degeneracy assumption, we only need to check regularity at the points $y$ on $H\cap \mcl{Y}_k$.
  Let $\sigma$ be the cone of $\Sigma$ such that $y$ lies on
the associated torus orbit $O(\sigma)$ in $Y$, and denote by $\gamma$ the face of $\Gamma(f)$ corresponding to $\sigma$.
We denote by  $M$ the characteristic monoid $M^{\sharp}_{\loga{Y},y}$ of $\loga{Y}$ at $y$.

If $y$ does not lie in the closure of $\mathrm{div}(f\circ h)\cap \mathbb{G}_{m,k}^n$, then locally around $y$, the log structure on $\loga{Y}$ is
the pullback of the natural log structure on the toric variety $Y$, which was described in Example \ref{exam:toric}. Thus $\loga{Y}$ is fine and saturated, because the divisorial log structure on $Y$ induced by the toric boundary has these properties.
 It also follows that $M=(\sigma^{\vee}\cap \Z^n)^{\sharp}$ so that $M^{\gp}$ has rank $r=\mathrm{dim}(\sigma)$. Locally at $y$, the maximal ideal of $M_{\loga{Y},y}$
 defines the toric orbit $O(\sigma)$, which is regular of codimension $r$ in $\mcl{Y}$. Thus $\loga{Y}$ is regular at $y$.

 Now suppose that $y$ lies in the schematic closure of $\mathrm{div}(f\circ h)\cap \mathbb{G}_{m,k}^n$.
   If $v$ is any lattice point on $\gamma+(\sigma^{\vee}\cap \Z^n)^{\times}$, then locally at $y$ this schematic closure is the zero locus of $f'=(f_{\gamma}/x^{v})+g$ where $g$ is an element of $\mathcal{O}_{\mcl{Y},y}$ that vanishes along $O(\sigma)$. The monoid $M$ is a submonoid of the monoid $\mathcal{O}_{\mcl{Y},y}/\mathcal{O}^{\times}_{\mcl{Y},y}$ of Cartier divisors on $\Spec \mathcal{O}_{\mcl{Y},y}$, generated by
   $(\sigma^{\vee}\cap \Z^n)^{\sharp}$ and $\mathrm{div}(f')$. The morphism of monoids
   $$ (\sigma^{\vee}\cap \Z^n)^{\sharp}\oplus \N\to M $$ that acts as the identity on the first summand and sends $(0,1)$ to  $\mathrm{div}(f')$ is an isomorphism.
    In particular, $M^{\gp}$ has rank $r=\mathrm{dim}(\sigma)+1$. Locally at $y$, the log scheme $\loga{Y}$ has a chart of the form $\loga{Y}\to \Spec^{\dagger}\Z[M]$ where the morphism of monoids $$M=(\sigma^{\vee}\cap \Z^n)^{\sharp}\oplus \N\to \mathcal{O}_{\mcl{Y},y}$$ maps $(m,1)$ to $\chi^m f'$. Here we denote by $\chi^m$ the pullback to $\mcl{Y}$ of the character of $Y$ associated with $m$. Since $M$ is fine and saturated, it follows that $\loga{Y}$ is fine and saturated, as well.

  Locally at $y$, the closed subscheme $Z$ of $\mcl{Y}$ defined by the maximal ideal of $M_{\loga{Y},y}$
 coincides with the schematic intersection of $O(\sigma)$ with $\mathrm{div}(f')$.
  Since $O(\sigma)$ is canonically isomorphic to $\Spec k[(\sigma^{\vee}\cap \Z^n)^{\times}]$ and $Z$ is the closed subscheme of $O(\sigma)$ defined by $f_{\gamma}/x^{v}$, the assumption that $f_{\gamma}$ has no critical points in $\mathbb{G}^n_{m,k}$ now implies that $Z$ is regular at $y$ of codimension $r$ in $\mathcal{Y}$. Hence, $\loga{Y}$ is regular at $y$.
  \end{proof}

In order to write down an explicit expression for the motivic zeta function $Z_f(T)$, we need to introduce some further notation.
 We set $X_0=f^{-1}(0)$. We define piecewise affine functions $N$ and $\nu$ on $\Sigma$ by setting
 \begin{eqnarray*}
 N(u)&=&\min\{u(m) \,|\,m \in \Gamma(f)\}
 \\ \nu(u)&=&u_1+\ldots+u_n
 \end{eqnarray*}
 for every $u$ in $\R^n$.
 For every face $\gamma$ of $\Gamma(f)$, we denote by $\sigma_{\gamma}$ the associated cone in the dual fan $\Sigma$ and by $\mathring{\sigma}_\gamma$ its relative interior.
  We write $O(\sigma_\gamma)$ for the torus orbit of $Y$ corresponding to $\sigma_\gamma$. We denote by $M_{\gamma}$ the
 fine and saturated monoid $\sigma_{\gamma}^{\vee}\cap \Z^n$.
 Since $\sigma_{\gamma}$ is contained in $(\R_{\geq 0})^n$, the dual cone $\sigma_\gamma^{\vee}$ contains $(\R_{\geq 0})^n$ and, in particular, the face $\gamma$. If $v$ is any point on the relative interior of $\gamma$, then the cone $\sigma_\gamma^{\vee}$ is generated by the vectors of the form $v'-v$ wit $v'$ in $\Gamma(f)$. 
    The group of invertible elements in $\sigma^{\vee}_{\gamma}$ coincides with the subspace of $\R^n$ generated by the translated face $\gamma-v$, and $M_{\gamma}^{\times}=(\sigma^{\vee}_{\gamma})^{\times}\cap \Z^n$. 
 Thus the image of $\gamma\cap \Z^n$ under the projection $M_\gamma\to M_\gamma^{\sharp}$ consists of a unique point, which we denote by $v_\gamma$.

 We write $X_{\gamma}(0)$ for the closed subscheme of $\mathbb{G}_{m,k}^n$ defined by $f_{\gamma}$, endowed with the trivial $\widehat{\mu}$-action. We view $X_{\gamma}(0)$ as a scheme over $X_0$ {\em via} the composition
 $$X_\gamma(0)\subset \Spec k[\Z^n]\to \Spec k[M_\gamma^{\times}]\cong O(\sigma_\gamma)\subset Y\to \A^n_k.$$
Note that the morphism
$X_\gamma(0)\to \A^n_k$ factors through $X_0$, because the image of $X_\gamma(0)$ in $\Spec k[M_\gamma^{\times}]$ is precisely the intersection of $O(\sigma_\gamma)$ with the strict transform of $\mathrm{div}(f)$, by the proof of Proposition \ref{prop:nondeg-smooth}.

 We also define an $X_0$-scheme $X_{\gamma}(1)$ with a good $\widehat{\mu}$-action, as follows. The element $v_\gamma$ of $M_\gamma^{\sharp}$ equals $0$ if and only if $O(\sigma_\gamma)$ is not contained in the zero locus of $f\circ h$. In that case, we set $X_\gamma(1)=\emptyset$.
  Otherwise, we can write $v_\gamma=\rho v_\gamma^{\mathrm{prim}}$ for a unique positive integer $\rho$ and a unique primitive vector $v_\gamma^{\mathrm{prim}}$ in $M_\gamma^{\sharp}$. We choose an element $w$ in $\Z^n$ such that $\langle m,w\rangle=\rho$ for every point $m$ on $\gamma$.
   We define $X_\gamma(1)$ to be the closed subscheme of $\mathbb{G}_{m,k}^n$ defined by the equation $f_\gamma=1$. We consider the left $\mu_{\rho}$-action on $\mathbb{G}^n_{m,k}$ with weight vector $w$:
   $$\zeta\ast (x_1,\ldots,x_n)=(\zeta^{w_1}x_1,\ldots,\zeta^{w_n}x_n).$$ The subscheme $X_{\gamma}(1)$ is stable under this action, and thus inherits a left $\mu_{\rho}$-action from $\mathbb{G}^n_{m,k}$.
    We again view $X_\gamma(1)$ as an $X_0$-scheme {\em via} the composition
    $$X_\gamma(1)\subset \Spec k[\Z^n]\to \Spec k[M_\gamma^{\times}]\cong O(\sigma_\gamma)\to X_0.$$ The resulting morphism
    $X_\gamma(1)\to X_0$ is $\mu_{\rho}$-equivariant with respect to the trivial action on $X_0$.
  In formula \eqref{eq:nondeg} in the proof of Theorem \ref{thm:nondeg}, we will give an alternative expression for the class of $X_\gamma(1)$ in $K^{\widehat{\mu}}_0(\Var_{X_0})$ that implies that this class does not depend on the choice of the weight vector $w$.

  \begin{thm}\label{thm:nondeg}
  Let $f$ be a non-constant polynomial in $k[x_1,\ldots,x_n]$
  such that $f$ vanishes at the origin and $f$ is non-degenerate with respect to its Newton polyhedron $\Gamma(f)$.
  Set $X_0=f^{-1}(0)$.
Then the motivic zeta function $Z_f(T)$ can be written as
$$\sum_{\gamma}\left(\left([X_{\gamma}(0)]\frac{\LL^{-1}T}{1-\LL^{-1}T}+[X_{\gamma}(1)]\right)\sum_{u\in \mathring{\sigma}_{\gamma}\cap \N^n}\LL^{-\nu(u)}T^{N(u)}\right)$$
in $\mathcal{M}^{\widehat{\mu}}_{X_0}\llb T \rrb$, where the sum is taken over all the faces $\gamma$ of $\Gamma(f)$.
  \end{thm}
  \begin{proof}
  We will explain how this can be interpreted as a special case of Theorem \ref{thm:main}.
     We set $$\omega=\frac{dx_1\wedge \ldots \wedge dx_n}{df},$$
   viewed as a volume form on $\mcl{Y}_K$.

 We denote by $D$ the logarithmic boundary divisor on $\loga{Y}$. By the definition of a non-degenerate polynomial, the divisor
 $\mathrm{div}(f\circ h)\cap \mathbb{G}^n_{m,k}$ in the dense torus of $Y$ is smooth. We denote by $D'$ the restriction to $\mcl{Y}$ of the schematic closure of this divisor in $Y$.
 Then $D$ is the sum of $D'$
 and the restriction to $\mcl{Y}$ of the toric boundary on $Y$.
  By the description of the logarithmic strata on a regular log scheme in Section \ref{sec:fans},
the logarithmic strata of $\loga{Y}$ are precisely the sets $O(\sigma)\setminus D'$ and
the connected components of $O(\sigma)\cap D'$, where $\sigma$ ranges through the cones of $\Sigma$.
 Like on any regular log scheme, the points of the fan $F(\loga{Y})$ are the generic points of these logarithmic strata.


      Let $\gamma$ be a face of $\Gamma(f)$, and denote by $\sigma_{\gamma}$ the corresponding cone in the dual fan $\Sigma$. The points in $F(\loga{Y})_k$ that lie on $O(\sigma_{\gamma})$ are the generic points of $O(\sigma_{\gamma})\cap D'$ and, provided that $v_\gamma\neq 0$, also the generic point of  $O(\sigma_{\gamma})$.
 Let $\tau$ be a point of $F(\loga{Y})_k\cap O(\sigma_{\gamma})$ and set $M=M^{\sharp}_{\loga{Y},\tau}$. As usual, we write $e_t$ for the class of $t=f\circ h$ in $M$.

If $\tau$ lies on $D'$ then we have $M\cong M_\gamma^{\sharp}\oplus \N$ and this yields an isomorphism $$M^{\vee,\loc}\cong  (\mathring{\sigma}_{\gamma}\cap \N^n)\oplus (\N\setminus \{0\}).$$
 The element $e_t$ of $M$ corresponds to $(v_\gamma,1)$, because $D'$ is reduced. It follows that the root index of $\N\to M:1\mapsto e_t$ equals $1$, so that $\widetilde{E}(\tau)^o=E(\tau)^o$. This is the connected component of $O(\sigma_{\gamma})\cap D'$ that contains $\tau$. By the scissor relations in the Grothendieck ring, we have
 $$[O(\sigma_\gamma)\cap D']=\sum_{\tau \in F(\loga{Y})\cap O(\sigma_\gamma)\cap D'}[E(\tau)^o]$$ in $K_0(\Var_{X_0})$.

 We have seen in the proof of Proposition \ref{prop:nondeg-smooth} that $X_\gamma(0)$ is isomorphic to $(O(\sigma_{\gamma})\cap D')\times_k \mathbb{G}^{\mathrm{dim}(\sigma_\gamma)}_{m,k}$, so that
 $$[X_\gamma(0)]=[O(\sigma_\gamma)\cap D'](\LL-1)^{\dim(\sigma_\gamma)}$$ in $K_0(\Var_{X_0})$.
  Moreover, for every element $u'=(u,n)$ in $M^{\vee,\loc}$ we have $u'(e_t)=u(v_\gamma)+n=N(u)+n$ and $u'(\omega)=\nu(u)-N(u)$.
  Thus the contribution of $F(\loga{Y})_k\cap O(\sigma_{\gamma})\cap D'$ to the formula for $Z_f(T)=Z^{\widehat{\mu}}_{\mcl{Y},\omega}(\LL^{-1} T)$ in Theorem \ref{thm:main} is equal to
 $$[X_{\gamma}(0)]\frac{\LL^{-1}T}{1-\LL^{-1}T}\sum_{u\in \mathring{\sigma}_{\gamma}\cap \N^n}\LL^{-\nu(u)}T^{N(u)}.$$

  If $\tau$ does not lie on $D'$, then 
$E(\tau)^o=O(\sigma_{\gamma})\setminus D'$, and   
  $M$ is canonically isomorphic to $M_\gamma^{\sharp}$ so that we can identify $M^{\vee,\loc}$ with $\mathring{\sigma}_{\gamma}\cap \N^n$. The element $e_t$ of $M$ is equal to $v_\gamma$, so that $u(e_t)=u(v_\gamma)=N(u)$ for every $u$ in $\sigma_{\gamma}\cap \N^n$. We also have $u(\omega)=\nu(u)-N(u)$.
   Thus, in order to match the formula in the statement of the theorem with the one in Theorem \ref{thm:main}, it suffices to show that
   \begin{equation}\label{eq:nondeg}
   [X_\gamma(1)]=[\widetilde{E}(\tau)^o](\LL-1)^{n-\mathrm{dim}(\gamma)-1}
    \end{equation}
    in $K^{\widehat{\mu}}_0(\Var_{X_0})$. We write $v_\gamma=\rho v_\gamma^{\mathrm{prim}}$ for a positive integer $\rho$ and a  primitive vector $v_\gamma^{\mathrm{prim}}$ in $M_\gamma^{\sharp}$. Then $\rho$ is the root index of the morphism of monoids $\N\to M_\gamma^{\sharp}$ that maps $1$ to $e_t$.
 The torus orbit $O(\sigma_{\gamma})$ is canonically isomorphic to $\Spec k[M_\gamma^{\times}]$, and,  
  locally at every point of $O(\sigma_\gamma)$, we can write $f\circ h$ as
 $x^{v}((f_{\gamma}/x^{v})+g)$ where $v$ is any lattice point on $\gamma+(\sigma^{\vee}_\gamma)^{\times}$ and $g$ is a regular function on $\mcl{Y}$ that vanishes along $O(\sigma_\gamma)$. The intersection $O(\sigma_{\gamma})\cap D'$ is the zero locus of $f_{\gamma}/x^{v}$.
      We can choose $v$ in such a way that it is divisible by $\rho$, because $v_\gamma$ is divisible by $\rho$.
  Now it easily follows from the definition that $\widetilde{E}(\tau)^o$ is the cover of $E(\tau)^o$ defined by taking a $\rho$-th root of the unit $f_{\gamma}/x^{v}$:
 $$\widetilde{E}(\tau)^o\cong \Spec k[M_\gamma^{\times},T,T^{-1}]/((f_\gamma/x^{v})-T^{\rho}).$$ The group scheme $\mu_\rho$ acts on $\widetilde{E}(\tau)^o$ from the left by the inverse of multiplication  on $T$, that is, $T \ast \zeta=\zeta^{-1}T$.

 We
 define a $\mu_\rho$-equivariant morphism of $X_0$-schemes $X_\gamma(1)\to \widetilde{E}(\tau)^o$ by means of
 the morphism of $k$-algebras
 $$ k[M_\gamma^{\times},T,T^{-1}]/((f_\gamma/x^{v})-T^{\rho})\to k[\Z^n]/(f_\gamma-1)$$ that maps $T$ to $x^{-v/\rho}$ and
  that maps $x^m$ to itself, for every $m\in M_\gamma^{\times}$.
   The morphism $X_\gamma(1)\to \widetilde{E}(\tau)^o$ is an equivariant torsor
   with translation group $\Spec \Z[V^{\bot}_{\gamma}\cap \Z^n]$, where $V_{\gamma}$ is the sub-vector space of $\R^n$ generated by $\gamma$, and $V^{\bot}_{\gamma}$ denotes the orthogonal subspace. 
   The equality \eqref{eq:nondeg} now follows from Proposition \ref{prop:eqtor}.
    \end{proof}

\begin{rmk}\label{rem:guibert}
The calculation of $Z_f(T)$ in \cite[\S2.1]{guibert} contains the following flaws:
 the $\widehat{\mu}$-action on the schemes $X_\gamma(1)$ is ill-defined; the term involving $[X_\gamma(1)]$ should be omitted if $v_\gamma=0$; the factor $(\LL-1)$ after $[X_\gamma(0)]$ should be omitted; the $X_0$-scheme structure on $X_\gamma(0)$ and $X_\gamma(1)$ is not specified.
\end{rmk}

  \begin{cor}\label{cor:nondeg}
    Let $f$ be a non-constant polynomial in $k[x_1,\ldots,x_n]$
  such that $f$ vanishes at the origin and $f$ is non-degenerate with respect to its Newton polyhedron $\Gamma(f)$.
   Denote by $R(f)$ the set of primitive generators of the rays of the dual fan of $\Gamma(f)$ (that is, the inward pointing primitive normal vector on the facets of
   $\Gamma(f)$).
Then the set
$$\{-1\}\cup \{-\frac{\nu(u) }{N(u)}\,|\,u\in R(f), \,N(u)\neq 0 \}$$
is a set of candidate poles of $Z_f(T)$.
  \end{cor}
  \begin{proof}
  This follows from Theorem \ref{thm:nondeg} in the same way as in the proof of Proposition \ref{prop:poles}.
  \end{proof}

 Note that this set of candidate poles is substantially smaller than the set of candidates we would get from a toric log resolution of $(\A^n_k,X_0)$: the latter set would include
 the candidate poles associated with all the rays in a regular subdivision of the dual fan of $\Gamma(f)$. An analogous result for Igusa's $p$-adic zeta function was proven in \cite{denef-hoorn}.

 The same method of proof yields similar results for the local motivic zeta function $Z_{f,O}(T)$ of $f$ at the origin $O$ of $\A^n_k$.
  This zeta function is defined as
 the image of $Z_f(T)$ under the base change morphism $\mathcal{M}^{\widehat{\mu}}_{X_0}\to \mathcal{M}^{\widehat{\mu}}_{O}=\mathcal{M}^{\widehat{\mu}}_{k}$.
 In fact, we only need to
 assume that $f$ is non-degenerate with respect to the {\em compact} faces of its Newton polyhedron. This means that for every  compact face $\gamma$ of $\Gamma(f)$, the polynomial $f_\gamma$ has no critical points in $\mathbb{G}^n_{m,k}$.

 \begin{thm}
  We keep the notations of Theorem \ref{thm:nondeg}, but we replace the non-degeneracy assumption on $f$ by the weaker condition that
 $f$ is non-degenerate with respect to the compact faces of $\Gamma(f)$. Let $O$ be the origin of $\A^n_k$. Then the motivic zeta function $Z_{f,O}(T)$ of $f$ at $O$ can be written as
$$\sum_{\gamma}\left(\left([X_{\gamma}(0)]\frac{\LL^{-1}T}{1-\LL^{-1}T}+[X_{\gamma}(1)]\right)\sum_{u\in \mathring{\sigma}_{\gamma}\cap \N^n}\LL^{-\nu(u)}T^{N(u)}\right)$$
in $\mathcal{M}^{\widehat{\mu}}_{k}$, where the sum is taken over all the compact faces $\gamma$ of $\Gamma(f)$.
 \end{thm}
 \begin{proof}
  The non-degeneracy condition on $f$
 guarantees that $\loga{Y}$ is smooth over $S^\dagger$ at every point of $h^{-1}(O)$, by the same arguments as in the proof of Proposition \ref{prop:nondeg-smooth}.
 The remainder of the argument is identical to the proof of Theorem \ref{thm:nondeg}: we only need to take into account that $O(\sigma_\gamma)$
 lies in $h^{-1}(O)$ if  $\gamma$ is compact, and has empty intersection with $h^{-1}(O)$ otherwise.
\end{proof}

\begin{rmk}
 The monodromy conjecture for non-degenerate polynomials in at most $3$ variables has been proven  for the topological zeta function \cite{lema} and the $p$-adic and na\"\i ve motivic zeta functions \cite{bories} (in a weaker form, replacing roots of the Bernstein polynomial by local monodromy eigenvalues). See also \cite{loeser-nondeg} for partial results in arbitrary dimension in the $p$-adic setting.
\end{rmk}


\begin{thebibliography}{SGAIII00}

\bibitem[Be96]{berk}
V.~G. Berkovich.
\newblock {Vanishing cycles for formal schemes II.}
\newblock {\em Invent. Math.} 125(2):367--390, 1996.

\bibitem[BV16]{bories}
B.~Bories and W.~Veys,
\newblock
Igusa's $p$-adic local zeta function and the monodromy conjecture for non-degenerate surface singularities.
\newblock {\em Mem. Amer. Math. Soc.} 242, 2016.

\bibitem[BLR90]{BLR}
S.~Bosch, W.~{L\"u}tkebohmert, and M.~Raynaud.
\newblock {\em {N\'eron models.}} Volume~21 of
 {\em Ergebnisse der Mathematik und ihrer Grenzgebiete}. \newblock Springer-Verlag, 1990.

\bibitem[BM17]{BM}
M.~Brown and E.~Mazzon.
\newblock {The essential skeleton of a product of degenerations.}
\newblock Preprint, arXiv:1712.07235.

\bibitem[Bu15a]{PhD}
E.~Bultot.
\newblock {\em Motivic Integration and Logarithmic Geometry}.
\newblock PhD thesis, KU Leuven 2015, arXiv:1505.05688.

\bibitem[Bu15b]{cras}
E.~Bultot.
\newblock Computing zeta functions on log smooth models.
\newblock {C. R. Math. Acad. Sci. Paris}, 353(3):261--264, 2015.


\bibitem[DH01]{denef-hoorn}
J.~Denef and K.~Hoornaert.
\newblock Newton polyhedra and {I}gusa's local zeta function.
\newblock {\em J. Number Theory} 89(1):31--64, 2001.

\bibitem[DL01]{DL-barc}
J.~Denef and F.~Loeser.
\newblock Geometry on arc spaces of algebraic varieties.
\newblock In: {\em {E}uropean {C}ongress of {M}athematics, {V}ol. I (Barcelona, 2000)}. Vol.~201 of {\em Progr. Math.}, Birkh\"auser, Basel, pages 327--348, 2001.

\bibitem[EHN15]{logjumps}
D.~Eriksson, L.H.~Halle and J.~Nicaise.
\newblock{ A logarithmic interpretation of {E}dixhoven's jumps for {J}acobians.}
\newblock{\em Adv. Math.}  279:532--574, 2015.

\bibitem[Fu93]{fulton}
W.~Fulton
\newblock {\em Introduction to toric varieties.}
 Volume~131 of {\em Annals of Mathematics Studies.}
 \newblock Princeton University Press, Princeton, NJ, 1993.

\bibitem[GR15]{GaRa}
O.~Gabber and L.~Ramero.
\newblock{Foundations for almost ring theory -- Release 6.9.}
\newblock{Preprint}, arXiv:math/0409584v10, 2015.

\bibitem[Gu02]{guibert}
G.~Guibert.
\newblock Espaces d'arcs et invariants d'{A}lexander.
\newblock {\em Comment. Math. Helv.} 77(4):783--820, 2002.

\bibitem[HN11]{HaNi}
L.H.~Halle and J.~Nicaise.
\newblock{Motivic zeta functions of abelian varieties, and the monodromy conjecture.}
\newblock{\em Adv.~Math.}, 227:610--653, 2011.

\bibitem[HN16]{HaNi-CY}
L.H.~Halle and J.~Nicaise.
\newblock{Motivic zeta functions of degenerating Calabi-Yau varieties.}
\newblock To appear in {Math.~Ann.,}  arXiv:1701.09155.

\bibitem[Ha15]{hartmann}
A.~Hartmann.
\newblock Equivariant motivic integration on formal schemes and the motivic zeta function.
\newblock{\em Preprint}, 	arXiv:1511.08656.

\bibitem[Ig00]{igusa}
J.-i.~Igusa.
\newblock An introduction to the theory of local zeta functions.
\newblock Vol. ~14 of {\em AMS/IP Studies in Advanced Mathematics.} Amer.~Math.~Soc., Providence, RI; International Press, Cambridge, MA, 2000.

\bibitem[IS15]{quotient}
H.~Ito and S.~Schr{\"o}er.
\newblock Wild quotient surface singularities whose dual graphs are not star-shaped.
\newblock {\em Asian J. Math.} 19(5):951--986, 2015.

\bibitem[Ka89]{kato-log}
K.~Kato. \newblock Logarithmic structures of Fontaine-Illusie.
\newblock In: {\em Algebraic analysis, geometry, and number theory}. Johns Hopkins Univ. Press, Baltimore, MD, pages 191--224, 1989.

\bibitem[Ka94]{kato}
K.~Kato.
\newblock{Toric singularities.}
\newblock {\em Am. J. Math.}, 116(5):1073--1099, 1994.

\bibitem[KKMS73]{KKMS}
G.~Kempf, F.~Knudsen, D.~Mumford, and B.~Saint-Donat.
\newblock {\em Toroidal embeddings 1}. Volume 339 of {\em Lecture Notes in
  Mathematics}.
\newblock Springer-Verlag, 1973.

\bibitem[Ko76]{kouchnirenko}
A.G.~Kouchnirenko.
\newblock Poly\`edres de {N}ewton et nombres de {M}ilnor.
\newblock {\em Invent. Math.}, 32:1--32, 1976.

\bibitem[LVP11]{lema}
A.~Lemahieu and L.~Van Proeyen.
\newblock  Monodromy conjecture for nondegenerate surface singularities.
\newblock {\em Trans. Amer. Math. Soc.} 363(9):4801--4829, 2011.

\bibitem[Lo88]{loeser}
F.~Loeser.
\newblock Fonctions d'{I}gusa $p$-adiques et polyn{\^o}mes de {B}ernstein.
\newblock {\em Amer. J. Math.,} 110(1):1--21, 1988.

\bibitem[Lo90]{loeser-nondeg}
F.~Loeser.
\newblock Fonctions d'Igusa $p$-adiques, polyn\^omes de Bernstein, et poly\`edres de Newton.
\newblock {\em J. Reine Angew. Math.} 412:75--96, 1990.

\bibitem[LS03]{motrigid}
F.~Loeser and J.~Sebag.
\newblock {Motivic integration on smooth rigid varieties and invariants of
  degenerations}.
\newblock {\em Duke Math. J.}, 119:315--344, 2003.

\bibitem[Na97]{nakayama}
C.~Nakayama.
\newblock{Logarithmic {\'e}tale cohomology.}
\newblock {\em Math. Ann.}, 308:365--404, 1997.

\bibitem[Ni09]{Ni}
J.~Nicaise.
\newblock A trace formula for rigid varieties, and motivic Weil generating
  series for formal schemes.
\newblock {\em Math. Ann.}, 343(2):285-349, 2009.

\bibitem[Ni13]{Ni-tameram}
J.~Nicaise.
\newblock {Geometric criteria for tame ramification.}
\newblock {\em Math. Z.}, 273(3):839--868, 2013.


\bibitem[NOR16]{NOR}
J.~Nicaise, D.P.~Overholser and H.~Ruddat.
\newblock Motivic zeta functions of the quartic and its mirror dual.
\newblock
In: {\em String-Math 2014}, volume 93 of {\em Proceedings of Symposia in Pure Mathematics}, AMS, pages 187--198, 2016.

\bibitem[NS07]{NiSe}
J.~Nicaise and J.~Sebag.
\newblock The motivic {S}erre invariant, ramification, and the analytic
  {M}ilnor fiber.
\newblock {\em Invent. Math.}, 168(1):133-173, 2007.


\bibitem[NS11]{NS-K0}
J.~Nicaise and J.~Sebag.
\newblock{The Grothendieck ring of varieties.} In: R.~Cluckers,
J.~Nicaise and J.~Sebag (editors). {\em Motivic integration and
its interactions with model theory and non-archimedean geometry.}
 Volume 383 of {\em London Mathematical Society Lecture Notes
 Series}.
Cambridge University Press, pages 145--188, 2011.

\bibitem[Ni06]{niziol}
W.~Niziol.
\newblock Toric singularities: log-blow-ups and global resolutions.
\newblock {\em J. Algebraic Geom.} 15(1):1--29, 2006.

\bibitem[Ro04]{rodrigues}
B.~Rodrigues.
\newblock {On the monodromy conjecture for curves on normal surfaces.}
\newblock {\em Math. Proc. Camb. Philos. Soc.}, 136(2):313--324, 2004.

\bibitem[RV03]{RoVe}
B.~Rodrigues and W.~Veys.
\newblock Poles of zeta functions on normal surfaces.
\newblock {\em Proc. London Math. Soc. (3)}, 87(1):164--196, 2003.

\bibitem[Sa04]{saito}
T.~Saito.
\newblock Log smooth extension of a family of curves and semi-stable reduction.
\newblock {\em J. Algebraic Geom.}  13(2):287--321, 2004.

\bibitem[St83]{strauss}
L.~Strauss.
\newblock Poles of a two variable p-adic complex power.
\newblock {\em Trans. Amer. Math. Soc.},
278(2):481--493, 1983.



\bibitem[Ve97]{veys}
W.~Veys.
\newblock Zeta functions for curves and log canonical models.
\newblock {\em Proc. London Math. Soc.}, 74(2):360--378, 1997.

\bibitem[Vi04]{vidal}
I.~Vidal.
\newblock Monodromie locale et fonctions z\^eta des log sch\'emas.
\newblock In: {\em Geometric aspects of Dwork theory. Vol. II.}  Walter de Gruyter GmbH \& Co. KG, Berlin, pages 983--1038, 2004.

\bibitem[Wan97]{wang}
J.~Wang.
\newblock {\em Equivariant resolution of singularities and semi-stable
  reduction in characteristic zero}.
\newblock PhD thesis, Massachusetts Institute of Technology, 1997.

\end{thebibliography}
\end{document}